\documentclass[amscd,amssymb,11pt]{amsart}

\usepackage{amssymb}
\usepackage{graphicx}
\usepackage{xcolor}
\usepackage[all]{xy}

\newtheorem{thm}{Theorem}[section]
\newtheorem{lem}[thm]{Lemma}
\newtheorem{cor}[thm]{Corollary}
\newtheorem{prop}[thm]{Proposition}

\theoremstyle{definition}

\newtheorem{defn}[thm]{Definition}
\newtheorem{defns}[thm]{Definitions}

\newtheorem{ex}[thm]{Example}
\newtheorem{exs}[thm]{Examples}

\newtheorem{problem}[thm]{Problem}

\theoremstyle{remark}

\newtheorem{rem}[thm]{Remark}
\newtheorem{rems}[thm]{Remarks}

\numberwithin{equation}{section}

\newcommand{\thmref}[1]{Theorem~\ref{#1}}
\newcommand{\corref}[1]{Corollary~\ref{#1}}
\newcommand{\secref}[1]{\S\ref{#1}}

\newcommand{\propref}[1]{Proposition~\ref{#1}}
\newcommand{\lemref}[1]{Lemma~\ref{#1}}

\newcommand{\exref}[1]{Example~\ref{#1}}
\newcommand{\remref}[1]{Remark~\ref{#1}}

\newcommand{\Hom}{\operatorname{Hom}}

\newcommand{\Res}{\operatorname{Res}}

\newcommand{\im}{\operatorname{im}}

\newcommand{\Sp}{{\mathcal  S}}

\newcommand{\Pp}{{\mathcal  P}}

\newcommand{\Z}{{\mathbb  Z}}
\newcommand{\F}{{\mathbb  F}}
\newcommand{\Q}{{\mathbb  Q}}
\newcommand{\N}{{\mathbb  N}}
\newcommand{\R}{{\mathbb  R}}
\newcommand{\C}{{\mathbb  C}}

\newcommand{\RP}{{\mathbb  R \mathbb P}}

\newcommand{\Sinfty}{\Sigma^{\infty}}

\newcommand{\sm}{\wedge}

\newcommand{\ra}{\rightarrow}
\newcommand{\xra}{\xrightarrow}

\newcommand{\hra}{\hookrightarrow}

\begin{document}

\title[Chromatic Fixed Point Theory]{Chromatic Fixed Point Theory and the Balmer spectrum for extraspecial 2-groups.}

\author[Kuhn]{Nicholas J.~Kuhn}
\email{njk4x@virginia.edu}

\address{Department of Mathematics \\ University of Virginia \\ Charlottesville, VA 22903}

\author[Lloyd]{Christopher J.~R.~Lloyd}
\email{cjl8zf@virginia.edu}



\date{July 31, 2020. Revised July 22, 2022.}


\begin{abstract}

In the early 1940's, P.A.Smith showed that if a finite $p$--group $G$ acts on a finite dimensional complex $X$ that is mod $p$ acyclic, then its space of fixed points, $X^G$, will also be mod $p$ acyclic.

In their recent study of the Balmer spectrum of equivariant stable homotopy theory,  Balmer and Sanders were led to study a question that can be shown to be equivalent to the following: if a $G$--space $X$ is a equivariant homotopy retract of the $p$-localization of a based finite $G$--C.W.~ complex, given $H<G$ and $n$, what is the smallest $r$ such that if $X^H$ is acyclic in the $(n+r)$th Morava $K$--theory, then $X^G$ must be acyclic in the $n$th Morava $K$--theory?  Barthel et.~ al.~ then answered this when $G$ is abelian, by finding general lower and upper bounds for these `blue shift' numbers which agree in the abelian case.

In our paper, we first prove that these potential chromatic versions of Smith's theorem are equivalent to chromatic versions of a 1952 theorem of E.E.Floyd, which replaces acyclicity by bounds on dimensions of mod $p$ homology, and thus applies to all finite dimensional $G$--spaces.  This unlocks new techniques and applications in chromatic fixed point theory.

Applied to the problem of understanding blue shift numbers, we are able to use classic constructions and representation theory to search for lower bounds.  We give a simple new proof of the known lower bound theorem, and then get the first results about nonabelian 2-groups that don't follow from previously known results.  In particular, we are able to determine all blue shift numbers for extraspecial 2-groups.

Applied in ways analogous to Smith's original applications, we prove new fixed point theorems for $K(n)_*$--homology disks and spheres.

Finally, our methods offer a new way of using equivariant results to show the collapsing of certain Atiyah-Hirzebruch spectral sequences in certain cases. Our criterion appears to apply to the calculation of all 2-primary Morava $K$--theories of all real Grassmanians.
\end{abstract}

\maketitle


\section{Introduction} \label{introduction}

If a finite group $G$ acts on a space $X$, we let $X^G$ denote its subspace of fixed points. We say that $X$ is based if it has a basepoint fixed by $G$.  We also recall that if $N < G$ is normal, then $X^N$ will be a $G/N$--space and $(X^N)^{G/N} = X^G$.

With a series of papers beginning with \cite{PA Smith 1938}, P.~A.~Smith introduced homological methods to study the structure of the fixed points of $G$--spaces, when $G$ is a $p$--group.  In particular, \cite[Theorem II]{PA Smith 1941} shows the following:

{\em Let $H$ be a subgroup of a finite $p$--group $G$, and let $X$ be a finite dimensional $G$--C.W.~ complex.
\begin{equation} \label{Smith thm} \text{If } H_*(X^H;\Z/p) \simeq \Z/p \text{ then } H_*(X^G;\Z/p) \simeq \Z/p.
\end{equation}}
Since there exists a normal series $H = K_0 \lhd \dots \lhd K_r = G$ with all subquotients cyclic of order $p$, an inductive argument shows that this theorem for all pairs $H < G$ is a consequence of the special case $\{e\} < C_p$.

Note that restricted to based $G$--spaces -- spaces for which a fixed point has already been assumed to exist, -- (\ref{Smith thm}) can be written as
{\em \begin{equation} \label{ based Smith thm} \text{If } \widetilde H_*(X^H;\Z/p) = 0 \text{ then } \widetilde H_*(X^G;\Z/p) =0.
\end{equation}}

A decade later, E.E.Floyd upgraded this result to one that gives information about all finite $G$--spaces, not just those that are $\Z/p$--acyclic.  \cite[Theorem 4.4]{floyd TAMS 52} shows the following: \\

{\em Let $H$ be a subgroup of a finite $p$--group $G$, and let $X$ be a finite dimensional $G$--C.W.~ complex such that $\dim_{\Z/p} H_*(X^H;\Z/p)$ is finite.
\begin{equation} \label{Floyd thm} \text{Then } \dim_{\Z/p} H_*(X^H;\Z/p) \geq \dim_{\Z/p} H_*(X^G;\Z/p).
\end{equation}}

As before, the theorem follows from the special case $\{e\} < C_p$. Note that this theorem restricted to based $G$--spaces {\em is} equivalent to the general case, as, when $X$ is unbased, the inequality for the based space $X_+$, the union of $X$ with a disjoint fixed basepoint, implies the inequality for $X$.

Also proved by Floyd was an Euler characteristic condition: with $G$, $H$, and $X$ as above,
{\em
\begin{equation} \label{Euler char thm} \chi(X^H) \equiv \chi(X^G) \mod p.
\end{equation}}
It is classical \cite[Theorem III.5.2]{bredon trans groups} that (\ref{Floyd thm}) and (\ref{Euler char thm}) imply the unbased form of Smith's theorem (\ref{Smith thm}). (See our proof of \thmref{unbased chom smith thm} for the same reasoning.)

Recently \cite{strickland,balmer sanders,6 author} the study of chromatic versions of Smith's theorem have arisen in work on the Balmer spectrum of the category of compact objects in the $G$--equivariant stable homotopy category, where $G$ is a finite group. In the key case when $G$ is a $p$--group, one is studying the tensor--triangulated ideals of compact objects in the $p$--local $G$--equivariant stable homotopy category.

To explain what these chromatic versions of Smith's theorem are, we need a definition.

\begin{defn} \label{p finite defn} Let $G$ be a finite $p$--group. A based $G$--space is a {\em $p$--local finite $G$-space} if it is a homotopy retract of the $p$--localization of a finite based $G$--C.W.~ complex.  (If $G$ is trivial, we will call such a space a {\em $p$--local finite space}.)
\end{defn}

We note that such spaces are $G$--spaces for which Floyd's theorem (\ref{Floyd thm}) applies: see \lemref{nice lemma}.

The connection with the work on the Balmer spectrum is via a lemma.

\begin{lem}\cite[Lemma 2.2]{barthel greenlees hausmann} \label{unstable to stable lemma} Let $G$ be a finite $p$--group. Compact objects in the $p$--local $G$--equivariant stable homotopy category are spectra $Y$ of the form $Y \simeq S^{-W}\Sinfty_GX$, where $W$ is a real representation of $G$, and $X$ is a $p$--local finite $G$--space as just defined.
\end{lem}

Now recall that, for each prime $p$, the Morava K-theories are a family of generalized homology theories equipped with products. $K(0)_*(X) = H_*(X;\Q)$ and, for $n \geq 1$, the theory $K(n)_*(X)$ has coefficient ring $K(n)_*$ equal to the graded field $\Z/p[v_n^{\pm 1}]$, with $|v_n| = 2p^n-2$, and satisfies a Kunneth theorem.  It is sometimes convenient to let $K(\infty)_*(X) = H_*(X;\Z/p)$.

We then have the following problem.
\begin{problem} \label{chro smith problems} Given a finite $p$--group $G$, and subgroup $H$, for what pairs $(n,m)$  is it true that for all $p$--local finite $G$--spaces $X$,
$$  \text{ if } \widetilde K(m)_*(X^H) = 0 \text{ then } \widetilde K(n)_*(X^G) = 0? $$
When this is true, we will say that {\em the $(G,H,n,m)$ Chromatic Smith Theorem} holds.
\end{problem}

Smith's theorem (\ref{ based Smith thm}) tells us that the $(G,H,\infty,\infty)$ Chromatic Smith Theorem holds for all $H \leq G$.  The existence of the Atiyah--Hirzebruch Spectral Sequence (AHSS) implies that if a space is acyclic in mod $p$ homology then it is $K(n)_*$--acyclic for all $n< \infty$, and thus the $(G,H,n,\infty)$ Chromatic Smith Theorem also holds for all finite $n$.

Nonequivariant work of Doug Ravenel and Steve Mitchell in the late 1970's and early 1980's gives us more information.

In \cite[Theorem 2.11]{ravenel conjs paper}, Ravenel proved the following:

{\em Let $X$ be a $p$--local finite space.
\begin{equation} \label{Ravenel thm}  \text{If } \widetilde K(m)_*(X) = 0 \text{ then } \widetilde K(n)_*(X) =0 \text{ for all } n<m.
\end{equation}}
We can conclude that the $(G,G,n,m)$ Chromatic Smith Theorem holds if $n\leq m$.

Mitchell \cite{mitchell} then showed that for each $n<\infty$, there exists a `type $n$' complex: a $p$--local finite space $X$ with $\widetilde K(n)_*(X) \neq 0$ but $\widetilde K(m)_*(X) = 0$ for all $m<n$.  Any type $n$ complex $X$ equipped with a trivial $G$--action yields an example for which $\widetilde K(n)_*(X^G) \neq 0$ but $\widetilde K(m)_*(X^H) = 0$ for all $m<n$ and $H \leq G$.
Thus the $(G,H,n,m)$ Chromatic Smith Theorem can possibly hold only if $n \leq m$.

The general problem can thus be regarded as the search for the correct common generalization of Ravenel's and Smith's theorems.  We note that, though normal series of the form $H = K_0 \lhd \dots \lhd K_r = G$ give some information from simpler cases, they don't necessarily give all.  Examples in this paper will make this very clear.

As we will review in \secref{balmer section}, given a finite $p$--group $G$, answering Problem \ref{chro smith problems} for all $(H,K,n,m)$ with $K < H < G $ is equivalent to the following problems in stable equivariant homotopy:
\begin{itemize}
\item  identifying the inclusions between prime ideals in the Balmer spectrum for the $p$--local $G$--equivariant stable homotopy category,
\item  characterizing which `chromatic type functions' can be realized by finite $G$--spaces or spectra.
\end{itemize}

The paper \cite{balmer sanders} made clear that these are deep and interesting problems to study, with solutions organized by the need to compute `blue shift' numbers, which they were able to compute when $G = C_p$.   Then in \cite{6 author}, these problems were solved  when $G$ is an abelian $p$--group, by finding general group theoretic upper and lower bounds for blue shift numbers which agree when $G$ is abelian.

In our paper, we first show that chromatic Smith theorems are surprisingly {\em equivalent} to chromatic versions of Floyd's theorem.  We look backwards to show this: our proof uses an idea of Jeff Smith critically used in a nonequivariant setting by Mike Hopkins and Smith \cite{hs}, in their mid 1980's proof of the Periodicity Theorem. Some special (modular) representations of the symmetric groups are used, as in \cite[Appendix C]{ravenel orange book} and \cite{jeff smith}.

The equivalence of the chromatic Smith and Floyd theorem problems then allows us to use classic constructions and (ordinary) representation theory to search for examples giving blue shift number lower bounds.  We give a simple new proof of the lower bound theorem of \cite{6 author}, and then get the first results about nonabelian groups that don't follow from previously known results.  In particular, we are able to answer Problem \ref{chro smith problems} for a family of 2-groups which include all extraspecial 2-groups.

One can also now apply chromatic Floyd theorems in situations where the corresponding chromatic Smith theorems have already been proved.  As a first application, we can deduce unbased versions of chromatic Smith theorems: these are results that cannot be deduced just using theorems in \cite{6 author}.  Similarly, we get results about group actions on $K(n)_*$--homology spheres, parallel to the classic case, illustrated with a rather down-to-earth result about involutions on the 5-dimensional Wu manifold.  Finally, we give a novel method using our Floyd theorems to show that, for some spaces $X$, the AHSS computing $K(n)_*(X)$ collapses after the first nonzero differential.  Again there are examples with familiar spaces:  in \cite{kuhn lloyd grassmanians}, summarized here in \exref{grassmanian calculation}, we show that our criteria apply to the calculation of $K(n)_*(Gr_d(\R^n))$ in many, and conjecturally all, cases.

In the next section, we will describe our main results in detail.

\subsection{Acknowledgements}

Our example resolving the open question about the Balmer spectrum when $G=D_8$ was first found in June, 2019, with a computer search by the second author using GAP. The second extraspecial 2--group example, with $G$ of order 32, was also found this way, which led to us discovering this infinite family of examples.  Tim Dokchister's lovely website GroupNames has been useful in thinking about these and other examples.

The $D_8$ example was presented by the first author in a talk in Oberwolfach in August, 2019 \cite{kuhn lloyd oberwolfach}.  At the same meeting, Markus Hausmann asked if a chromatic Floyd theorem might be true, and it eventually occurred to us that we had been working for awhile with an answer to this question, leading to our presentation here.

The first author is a PI of RTG NSF grant DMS-1839968, which partially supported the research of the second author.

\section{Main Results}

To describe our main results, it will be useful to introduce some notation: let $k_n(X) = \dim_{K(n)_*}\widetilde K(n)_*(X)$, when $X$ is a based $G$--space and this dimension is finite.  When $X$ is a $p$--local finite space, the numbers $k_0(X), k_1(X), k_2(X), \dots$ form a nondecreasing sequence: see \remref{floyd thm remarks}(b) below.  Furthermore, as the AHSS converging to $K(n)_*(X)$ clearly collapses at $E^2_{*,*} = H_*(X;K(n)_*)$ if the dimension of $X$ is less than $|v_n|+1 = 2p^n-1$, this sequence stabilizes at $k_{\infty}(X) = \dim_{\Z/p} H_*(X;\Z/p)$.

\subsection{Chromatic Floyd theorems}

The Floyd theorem analogue of our chromatic Smith theorem problem, Problem \ref{chro smith problems}, goes as follows.

\begin{problem} \label{chro Floyd problems} Given a finite $p$--group $G$, and subgroup $H$, for what pairs $(n,m)$  is it true that for all $p$--local finite $G$--spaces $X$,
$$  k_m(X^H)\geq k_n(X^G)? $$
When this is true, we will say that {\em the $(G,H,n,m)$ Chromatic Floyd Theorem} holds.
\end{problem}

The $(G,H,n,m)$ Chromatic Floyd Theorem clearly implies the $(G,H,n,m)$ Chromatic Smith Theorem.  Our perhaps surprising discovery is that they are, in fact, equivalent.

\begin{thm} \label{equivalence thm} If the $(G,H,n,m)$ Chromatic Smith Theorem is true then the $(G,H,n,m)$ Chromatic Floyd Theorem is true.
\end{thm}

In \secref{proof of smith = floyd thm}, we prove this theorem in its contrapositive form.  An outline of the argument, an adaption of J.~ Smith's construction \cite[Appendix C]{ravenel orange book} to our equivariant setting, goes roughly as follows.  If the $(G,H,n,m)$ Chromatic Floyd Theorem is not true, there exists a $p$--local finite $G$--space $X$ with $k_m(X^H)< k_n(X^G)$.  We then consider a stable wedge summand $F$ of the $G$--space $X^{\sm k}$ having the form $F=eX^{\sm k}$, where $k$ is (roughly) determined by the number $k_n(X^G)$, and $e$ is a well chosen idempotent in the group ring $\Z_{(p)}[\Sigma_k]$ of the symmetric group $\Sigma_k$.  The properties of $e$ allow us to show that $F^H$ is $K(m)_*$--acyclic but $F^G$ is not $K(n)_*$--acyclic, and $F$ thus witnesses that the $(G,H,n,m)$ Chromatic Smith Theorem is not true.

\begin{rems} \label{floyd thm remarks}  {\bf (a)}\  When $n=m=\infty$ we learn that the classical Smith Theorem implies the classical Floyd Theorem. This seems to be a new observation, but is perhaps of only modest interest, as modern expositions of Smith theory -- see \secref{smith theory background sec} -- tend to prove Floyd's result first.

{\bf (b)} \   When $G$ is trivial, the theorem combined with Ravenel's acyclicity result cited above implies that if $m \geq n \geq 0$ then $k_m(X)\geq k_n(X)$ for all $p$--local finite spectra.  In fact,  \cite[Thm. 2.11]{ravenel conjs paper} already shows that this stronger result holds.  But one does learn something new for a related result: with a proof quite different from Ravenel's, in \cite{bousfield K(n) equivs}, Pete Bousfield showed that if $X$ is {\em any} nonempty space such that $\widetilde K(m)_*(X) = 0$, then $\widetilde K(n)_*(X) = 0$ for all $m\geq n \geq 1$. Part of our argument, stated as \thmref{smash summand theorem}, applies to all spaces, and we learn that if  $X$ is any space such that $\dim_{K(m)_*}K(m)(X)$ is finite, then, for all $m\geq n \geq 1$, $\dim_{K(n)_*} K(n)_*(X)$ will also be finite, and $k_m(X) \geq k_n(X)$.

\end{rems}

\begin{rem} \label{auto remark}  Problem \ref{chro smith problems} and Problem \ref{chro Floyd problems} are really problems about equivalence classes of pairs $(G,H)$ with $H\leq G$, where $(G_1,H_1) \sim (G_2,H_2)$ if there exists an isomorphism $\alpha: G_1 \xra{\sim} G_2$ such that $\alpha(H_1) = H_2$.  For given a $G_2$--space $X_2$, one can let $X_1$ be the $G_1$--space which is $X_2$ with $G_1$--action given by the formula $g\cdot x = \alpha(g)x$ for all $g \in G_1$.  Then $X_1^{G_1}$ will be homeomorphic to $X_2^{G_2}$, and $X_1^{H_1}$ will be homeomorphic to $X_2^{H_2}$.  We conclude that the chromatic Smith theorem holds for $(G_1,H_1,n,m)$ if and only if it holds for $(G_2,H_2,n,m)$, and similarly for the chromatic Floyd theorem.  Specializing to the case when $G_1=G_2=G$, we learn that if $\alpha$ is an automorphism of $G$, then the chromatic Smith theorem holds for $(G,H,n,m)$ if and only if it holds for $(G,\alpha(H),n,m)$\footnote{In the stable setting of \cite{balmer sanders}, this reasoning shows that $Out(G)$ will act as a group of homeomorphisms on the Balmer spectrum for the $G$--equivariant stable homotopy category.}.
\end{rem}

\subsection{Application to the computation of blue shift numbers}   The results of Mitchell and Ravenel show that if the $(G,H,n,n+r)$ Chromatic Smith Theorem holds then
\begin{itemize}
\item $r \geq 0$ and
\item the $(G,H,n,n+r^{\prime})$ Chromatic Smith Theorem holds for all $r^{\prime}\geq r$.
\end{itemize}

Translated from the stable setting, \cite[Proposition 7.1]{balmer sanders} is Neil Strickland's beautiful result\footnote{The authors of \cite{balmer sanders} thank Strickland for the inspiration for their Proposition 7.1, which appears, with essentially the same proof, as Proposition 16.9 in unpublished notes \cite{strickland} by Strickland dating from 2010.  (See also \S 3 of\cite{joachimi}, which was posted in preprint form in 2015.)} that for all finite $n$, the $(C_p,\{e\},n,n+1)$ Chromatic Smith Theorem holds. From this, one can deduce that, for all $H\leq G$,  the $(G,H,n,n+r)$ Chromatic Smith Theorem holds if the index of $H$ in $G$ is $p^r$.

This allows us to make the next definition, following the lead in \cite{6 author}.

\begin{defn}  Let $r_n(G,H)$ be defined to be the smallest $r$ such that the $(G,H,n,n+r)$ Chromatic Smith Theorem is true, i.e.~ the smallest $r$ so that for all $p$--local finite $G$--spaces $X$,
$$  \text{ if } \widetilde K(n+r)_*(X^H) = 0 \text{ then } \widetilde K(n)_*(X^G) = 0. $$
\end{defn}

From the above, we see that $0 \leq r_n(G,H) \leq \log_p [G:H]$. The paper \cite{balmer sanders} also shows that $1 \leq r_n(C_{p}, \{e\})$ so that $r_n(C_p,\{e\})=1$. The paper \cite{6 author} then goes further and shows that $r \leq r_n((C_{p})^r, \{e\})$ and $r_n(C_{p^k}, \{e\}) \leq 1$.  As we will explain in \secref{r numbers section}, these two results together imply general group theoretic bounds for all $H<G$:
$$ r_{-}(G,H) \leq  r_n(G,H) \leq r_{+}(G,H),$$
where $r_{-}(G,H)$ and  $r_{+}(G,H)$ are defined as follows.

\begin{defns} Let $H$ be a subgroup of a finite $p$--group $G$.

{\bf (a)} \ Let $r_{-}(G,H) = \text{ rank } G/H\Phi(G)$.  Here $H\Phi(G)< G$ is the subgroup generated by $H$ and the Frattini subgroup $\Phi(G)$, so that $G/H\Phi(G)$ is the maximal elementary abelian $p$-group quotient of $G$ with $H$ in the kernel.

{\bf (b)} \ Let $r_{+}(G,H)$ be the minimal $r$ such that there exists a chain of subgroups $ H = K_0 \lhd K_1 \lhd \dots \lhd K_r = G$ with each $K_{i-1}$ normal in $K_i$ and  $K_i/K_{i-1}$ cyclic.
\end{defns}

The lower bound for $r_{-}(G,H)$ agrees with the upper bound $r_{+}(G,H)$ in some cases,  notably whenever $G$ is abelian, when both bounds equal the rank of $G/H$.  But they don't agree in general, as one sees already when $G=D_8$, the dihedral group of order 8, and $H = C$, where $C$ is any of the noncentral subgroups of order 2: $r_{-}(D_8,C) = 1$ while $r_{+}(D_8,C) = 2$.

Our main application of \thmref{equivalence thm} is to use it to give lower bounds for $r_n(G,H)$ that sometimes improve upon $r_{-}(G,H)$. Meanwhile, there is cautious hope that $r_n(G,H)$ always equals the upper bound $r_{+}(G,H)$. Perhaps this hope is a bit perverse: at the end of \secref{r numbers section}, we give some examples showing how badly the function $r_+(G,H)$ behaves; e.g., it is not always monotone in the variable $H$, and it is not always additive under products of pairs.

There is also a conjecture that the value of $r_n(G,H)$ is independent of $n$ -- a phenomenon seen so far in all examples, but missing any conceptual explanation.  Of course, a proof that $r_n(G,H)$ always equals $r_{+}(G,H)$ would prove this.

To possibly show that $r_n(G,H) = r_{+}(G,H)$ in general, one needs to improve the lower bound.  By definition, a lower bound $r \leq r_n(G,H)$ means that there exists a $p$--local finite $G$--space $X$ with $X^G$ not $K(n)_*$--acyclic, but with $X^H$ $K(n+r-1)_*$--acyclic.  These are hard to find in the literature; to show that $r \leq r_n((C_{p})^r, \{e\})$, the authors of \cite{6 author} found one family of examples in the work of Greg Arone and Kathryn Lesh \cite{arone lesh} (really going back to work of Mitchell \cite{mitchell}).

A corollary to \thmref{equivalence thm} offers an easier way forward.

\begin{cor} \label{lower bound method cor} To show that $r \leq r_n(G,H)$, it suffices to find a $p$--local finite $G$--space $X$ with
$  k_{n+r-1}(X^H) < k_n(X^G)$.
\end{cor}

As we now illustrate, applying this corollary to very classic constructions will allow us a much simplified proof, for all primes, of the existing lower bound found in \cite{6 author}, and then better bounds for an infinite family of new examples when $p=2$, including the $(D_8,C)$ example mentioned above.

\subsubsection{New proofs of old lower bounds at all primes} \label{old examples intro subsection}

If $\omega$ is a unitary representation of a finite group $G$, let $S(\omega)$ be the sphere of unit length vectors.  This is a $G$--space, and also a free $S^1$--space, where $S^1 \subset \C$ acts via scalar multiplication.  The actions by $G$ and $S^1$ commute, and we let $L_p(\omega) = S(\omega)/C_p$, where $C_p<S^1$ is the group of $p$th roots of 1.

Thus $L_p(\omega)$ will be a lens space with an unbased $G$--action. We will see that it is easy to analyze the fixed point space $L_p(\omega)^G$ and then to compute the size of its Morava $K$--theories.

We let $E_r$ denote the elementary abelian $p$--group $C_p^r$.

\begin{ex} \label{el ab example}  Let $\rho_r^{\C}$ denote the complex regular representation of $E_r$.  If we let $\omega = p^n\rho_r^{\C}$, then
$$  k_{n+r-1}(L_p(\omega)_+) = 2p^{n+r-1}$$
which is less than
$$  k_n(L_p(\omega)^{E_r}_+) = 2p^{n+r}. $$
\end{ex}

Details will be in \secref{examples section}.

Thanks to \corref{lower bound method cor}, we recover the lower bound from \cite{6 author}, with our elementary examples replacing the much more delicate and technical examples of \cite[Thm.2.2]{6 author}.

\begin{thm} \label{el ab thm}  $r \leq r_n(E_r, \{e\})$. \label{el ab lower bound theorem}
\end{thm}

As already mentioned above, in \secref{r numbers section} we will show that from this, one can deduce that $r_{-}(G,H) \leq r_n(G,H)$ for all $n$ and $H<G$.

\subsubsection{New lower bounds for the extraspecial $2$--groups} \label{new examples intro subsection}

Let $D_8$ be the dihedral group of order 8, and let $C<D_8$ be any one of the four noncentral subgroups of order 2. (These are all equivalent under automorphisms of $D_8$.)   Then $r_{-}(D_8,C)$ = 1 while $r_+(D_8,C) = 2$, and this is the simplest example for which the blue shift numbers $r_n(D_8,C)$ can not be determined by the results in \cite{6 author}.

This example turns out to fit into an infinite family of examples.  Let $\widetilde E_{2r}$ denote the central product of $r$ copies of $D_8$.  This group is the extraspecial $2$--group of order $2^{1+2r}$ associated to a quadratic form $q: E_{2r} \ra C_2$ of Arf invariant 0.  As such, it is a nonabelian central extension
$ C_2 \ra \widetilde E_{2r} \ra E_{2r}$.

The analogue of $C < D_8$ is then $W_r < \widetilde E_{2r}$ where $W_r$ is any elementary abelian subgroup of rank $r$ that does not contain the central $C_2$. (All such subgroups are equivalent.)

It isn't hard to check that $r_{-}(\widetilde E_{2r},W_r)=r$ and $r_{+}(\widetilde E_{2r},W_r)=r+1$ so that $r \leq r_n(\widetilde E_{2r},W_r) \leq  r+1$.
By tweaking the construction in the previous subsection, we show that blue shift numbers attain the upper bound.

The tweak is as follows.  If $\omega$ is now a real representation of a finite 2--group $G$, let $\RP(\omega)$ denote the associated projective space.  As before, it is easy to analyze the fixed point space $\RP(\omega)^G$, and then to compute the size of its Morava $K$--theories.

\begin{ex} \label{extraspecial 2 group ex}   Let $\widetilde \rho_{2r}$ be the real regular representation of $E_{2r}$, pulled back to $\widetilde E_{2r}$. This is the sum of the $2^{2r}$ distinct one dimensional real representations of $\widetilde E_{2r}$, and $\widetilde E_{2r}$ has one more irreductible real representation $\Delta_r$ of dimension $2^r$. If we let $\omega = 2^{n+1}\widetilde \rho_{2r} \oplus \Delta_r$, then
$$  k_{n+r}(\RP(\omega)^{W_r}_+) = 2^{n+1+2r}-2^r$$
which is less than
$$ k_n(\RP(\omega)^{\widetilde E_{2r}}_+) = 2^{n+1+2r}. $$
\end{ex}

Again invoking \corref{lower bound method cor}, this example has the following consequence.

\begin{thm} \label{extraspecial group thm} For all $n$, $r_n(\widetilde E_{2r},W_r) = r+1 = r_+(\widetilde E_{2r},W_r)$.
\end{thm}

A variant of this last example will prove the following.

\begin{thm} \label{extraspecial group product with el ab thm} For all $n$,
$$r_n(\widetilde E_{2r} \times E_s,W_r \times \{e\}) = r+s+1  = r_+(\widetilde E_{2r}\times E_s,W_r \times \{e\}).$$
\end{thm}

This last theorem suffices to allow us to deduce the following consequence.

\begin{thm} \label{good family thm} Let $G$ be any 2-group fitting into a central extension
$$ C_2 \ra G \ra E$$
with $E$ elementary abelian.   For all $K<H<G$, $r_n(H,K) = r_+(H,K)$ for all $n$.
\end{thm}

The details will be in \secref{ex special 2-groups section}.

This computation of these blue shift numbers resolves all open questions about the topology of the Balmer spectrum for any such $G$.

\subsubsection{A new general lower bound theorem for $2$--groups} \label{general lower bound subsection}

An analysis of our argument for the extraspecial 2-groups leads to a general theorem that improves the lower bound for $r_n(G,H)$ for many other groups too.

To state this, we need a little bit of notation.  Given a finite 2--group $H$, let $e_H \in \R[H]$ be the central idempotent
$$ e_H = \frac{1}{|\Phi(H)|}\sum_{h \in \Phi(H)} h.$$
If $\omega$ is a real representation of $H$, then $e_H\omega$ is the maximal direct summand of $\omega$ on which $\Phi(H)$ acts trivially, so can be viewed as a real representation of $H/\Phi(H)$: see the paragraph before \corref{fixed point cor}.

\begin{thm} \label{general lower bound theorem}  Let $H$ be a nontrivial proper subgroup of a finite 2--group $G$ such that $\Phi(H) = \Phi(G)\cap H$, or, equivalently, $H/\Phi(H) \ra G/\Phi(G)$ is monic.

If $G$ has an irreducible real representation $\Delta$ such that $e_H\Res^G_H(\Delta)$ is the regular real representation of $H/\Phi(H)$, then, for all $n$,
$$ r_n(G,H) \geq r_-(G,H) + 1.$$
\end{thm}

The proof is in \secref{examples section}.

\begin{ex}  If $H = W_r < \widetilde E_{2r}=G$, then $\Delta_r$ satisfies the hypothesis of the theorem.
\end{ex}

\begin{ex}  Let $G$ be the semidirect product $C_2^3\rtimes C_4$, with $C_4$ acting faithfully on $C_2^3$, the group with GAP label 32 \#6.  Let $H<G$ be the cyclic subgroup of order 4 which is GAP subgroup \#24.  Then $\Phi(H) = \Phi(G)\cap H$, and $G$ has three distinct irreducible real representations which satisfy the hypothesis of the theorem: two 2--dimensional ones that are pulled back from a quotient map $G \twoheadrightarrow D_8$, and one that is  faithful of dimension 4.

One computes that $r_-(G,H) = 1$ and $r_+(G,H)$ = 3.  The theorem then tells us that $r_n(G,H)$ is either 2 or 3.
\end{ex}

\subsection{Further application of \thmref{equivalence thm}}

Our applications of \thmref{equivalence thm} in the last two subsections use \thmref{equivalence thm} in its contrapositive form: `if the $(G,H,n,m)$ Chromatic Floyd Theorem is not true then the $(G,H,n,m)$ Chromatic Smith Theorem is not true'.

When combined with the upper bound $r_n(G,H) \leq r_{+}(G,H)$, the direct statement -- `if the $(G,H,n,m)$ Chromatic Smith Theorem is  true then the $(G,H,n,m)$ Chromatic Floyd Theorem is true' --  implies the following theorem.

\begin{thm} \label{cyclic chromatic floyd theorem}  If $X$ is a $p$--local finite $G$--space, then, for all $r \geq r_+(G,H)$
$$k_{n+r}(X^H) \geq k_n(X^G) \text{ for all } n.$$
\end{thm}

This has a variety of interesting applications.

\subsubsection{Application to $K(n)_*$--homology disks}

For this application, and our next, we need to first note that Floyd's Euler characteristic theorem (\ref{Euler char thm}) implies chromatic versions.  If $X$ is space such that $\dim_{K(n)_*}K(n)_*(X)$ is finite, let $\chi_n(X)$ be the Euler characteristic of the graded $K(n)_*$--vector space $K(n)_*(X)$.  It is a standard lemma that the Euler characteristic of a finite chain complex of finite dimensional vector spaces equals the Euler characteristic of the resulting graded homology groups.  Applying this to the AHSS, we learn that if $H_*(X;\Z/p)$ is finite dimensional then $\chi(X) = \chi_n(X)$.  Thus (\ref{Euler char thm}) implies the following.

\begin{prop} \label{chrom Euler char prop} If $G$ is a finite $p$--group, $H$ a subgroup, and $X$ a $p$--local finite $G$--space, then $ \chi_n(X^H) \equiv \chi_n(X^G) \mod p.$
\end{prop}

As mentioned in the introduction, Floyd's theorems (\ref{Floyd thm}) and (\ref{Euler char thm}) imply the unbased Smith theorem.  The reasoning showing this in \cite[Section III.5]{bredon trans groups} adapts to show that \thmref{cyclic chromatic floyd theorem} and \propref{chrom Euler char prop} imply an analogous unbased chromatic Smith theorem.

To explain this, we need to describe the unbased $G$--spaces that we will consider.

\begin{defn} Let $G$ be a finite $p$--group.  An unbased $G$--space $X$ is {\em admissible} if $X_+$ is a $p$--local finite $G$--space.
\end{defn}

\begin{thm} \label{unbased chom smith thm}
Let $H$ be a subgroup of a finite $p$--group $G$, and let $X$ be an admissible $G$--space. If $r \geq r_+(G,H)$ and $K(n+r)_*(X^H) \simeq K(n+r)_*$ then $K(n)_*(X^G) \simeq K(n)_*$.
\end{thm}
\begin{proof} By assumption, $k_{n+r}(X^H_+) = 1$. \thmref{cyclic chromatic floyd theorem} then implies that $k_{n}(X^G_+) \leq 1$, and so is 1 or 0. The possibility that $k_{n}(X^G_+)=0$ (i.e. $X^G = \emptyset$) is ruled out by \propref{chrom Euler char prop}.
\end{proof}

Specializing to the case when $G$ is cyclic, $H = \{e\}$, and $n=0$, we learn the following.

\begin{cor} Let $C$ be a finite cyclic $p$--group, and $X$ an admissible $C$--space.  If $X$ is acyclic in  mod $p$ $K$--theory, then $X^C$ will be rationally acyclic.
\end{cor}

\begin{ex}  This corollary would apply to any admissible cyclic action on a space homotopy equivalent to the cofiber of an unstable Adams map between mod $p^k$ Moore spaces, for any $p$ and $k$.
\end{ex}

\subsubsection{Application to $K(n)_*$--homology spheres}

Similar to the previous subsection, we get the chromatic analogues of classical theorem about actions on mod $p$--homology spheres.  A space $X$ is a $K(n)_*$--homology sphere if $\dim_{K(n)_*}K(n)_*(X) = 2$.

\begin{thm} \label{unbased sphere thm}
Let $H$ be a subgroup of a finite $p$--group $G$, and let $X$ be an admissible $G$--space. If $r \geq r_+(G,H)$ and $X^H$ is a $K(n+r)$--homology sphere then $X^G$ will be a $K(n)_*$--homology sphere, or possibly empty if $p=2$.
\end{thm}
\begin{proof} This is similar to the proof of \thmref{unbased chom smith thm}. We are assuming that $k_{n+r}(X^H_+) = 2$, so \thmref{cyclic chromatic floyd theorem} implies that $k_{n}(X^G_+) \leq 2$. \propref{chrom Euler char prop} shows that $k_{n}(X^G_+)=1$ can't happen, and $k_{n}(X^G_+)=0$ could only happen if $p=2$.
\end{proof}

Specializing, as before, to the case when $G$ is cyclic, $H = \{e\}$, and $n=0$, we learn the following.

\begin{cor} Let $X$ be a mod $p$ $K$--theory sphere. Then the fixed point space of any admissible action of a finite cyclic $p$--group on $X$ will be a rational homology sphere, or possibly empty if $p=2$.
\end{cor}

\begin{ex} \label{Wu manifold thm}  This corollary applies to the 5-dimensional Wu manifold $M = SU(3)/SO(3)$ to show that if $X$ has an admissible action of $C_2$, then $M^{C_2}$ will be a rational homology sphere.  The fact that $M$ is a mod 2 $K$--theory sphere is an easy consequence of the action of $Sq^1$ and $Sq^2$ on $H^*(M;\Z/2)$, using the AHSS.  The possibility that $M^{C_2}= \emptyset$, i.e.~ $M$ has a free involution, can then be easily ruled out: $M$ is not the boundary of a 6--dimensional compact manifold, but any closed manifold with a free involution {\em is} a boundary.
\end{ex}

\subsubsection{Application to the calculation of $K(n)_*(X)$ using the AHSS}

Let $n\geq 1$.  If $X$ is a based space, the (reduced) Atiyah--Hirzebruch spectral sequence converging to $\widetilde K(n)^*(X)$ has $E_2^{*,*} = \widetilde H^*(X;\Z/p)[v_n^{\pm 1}]$, and has first possible nonzero differential given by $d_{2p^n-1}(x) = v_nQ_n(x)$,  where $Q_n$ is the unique primitive in the mod $p$ Steenrod algebra of degree $2p^n-1$ and satisfies $(Q_n)^2=0$. (When $p=2$, $Q_0 = Sq^1$ and $Q_n = [Sq^{2^n},Q_{n-1}]$.)

It follows that, if $\dim_{\Z/p} H^*(X;\Z/p)$ is finite, and we let
$$ H(X;Q_n) = \frac{\ker \{Q_n: \widetilde H^*(X;\Z/p) \ra \widetilde H^*(X;\Z/p)\}}{\im \{Q_n: \widetilde H^*(X;\Z/p) \ra \widetilde H^*(X;\Z/p)\}},$$
and then let $k_{Q_n}(X) = \dim_{\Z/p} H(X;Q_n)$, then  $k_{Q_n}(X)$ will equal the dimension of the $2p^n$th page of the AHSS as a $K(n)^*$--vector space. Thus we get the upper bound $k_n(X) \leq k_{Q_n}(X)$.

Meanwhile, \thmref{cyclic chromatic floyd theorem} implies that the lower bound $ k_{n-1}(X^C) \leq k_n(X)$
if $X$ admits an action of a cyclic $p$--group $C$ making $X$ into a $p$--local finite $C$--space.

If the lower bound matches the upper bound, we get the conclusion of the next theorem.

\begin{thm}  \label{AHSS collapse theorem} Let $n\geq 1$.  Suppose a space $X$ admits an action of a cyclic $p$--group $C$ making $X$ into a $p$--local finite $C$--space,  such that
$ k_{n-1}(X^C) = k_{Q_n}(X)$.
Then the AHSS computing $K(n)^*(X)$ collapses at $E_{2p^n}^{*,*}$ and $k_n(X) = k_{Q_n}(X)$.
\end{thm}

\begin{ex} \label{grassmanian calculation} In \cite{kuhn lloyd grassmanians}, we show that this criterion can be used to calculate $k_n(Gr_d(\R^m)_+)$ in many, and conjecturally all, cases, using group actions constructed as in \secref{new examples intro subsection}.  (Here $p=2$.)  In particular,
\begin{itemize}
\item well chosen $m$--dimensional real representations of $C_2$ are used to show that $k_n(Gr_d(\R^m)_+) = \binom{m}{d}$ if $m \leq 2^{n+1}$, and
\item well chosen $m$--dimensional real representations of $C_4$ are used to show that $k_n(Gr_2(\R^m)_+) = \binom{2^{n+1}-\epsilon}{2} + l$ if  $m = 2^{n+1}+2l -\epsilon$ with $l>0$ and $\epsilon$ equal to 0 or 1.
\end{itemize}
\end{ex}

\subsection{Organization of the rest of the paper}  Section \ref{background section} has some needed background information about classical Smith theory, the equivariant stable category, and Morava K-theories.

In \secref{balmer section}, we translate our chromatic Smith theorem and Floyd theorem problems into the stable setting. The former is the problem whose solution answers remaining questions about the topology of the Balmer spectra of the equivariant stable categories studied in \cite{balmer sanders}.  We also show how this problem is equivalent to understanding what chromatic type functions can be topologically realized.  Much of this material is in \cite{balmer sanders, 6 author}, but we hope our exposition will make these interesting problems more accessible to those less conversant with stable equivariant homotopy theory.

In \secref{r numbers section}, we run through basic properties of the blue shift numbers $r_n(G,H)$ and the group theoretic lower and upper bounds $r_-(G,H)$ and $r_+(G,H)$, and how the two main results from \cite{6 author} are used.

In \secref{proof of smith = floyd thm} we prove \thmref{equivalence thm}, after reviewing Jeff Smith's construction which uses certain idempotents in the group rings of the symmetric groups.

In \secref{examples section} we provide details of how we use real and complex representations to construct examples that give lower bounds for $r_n(G,H)$. We illustrate this with the details of \exref{el ab example}, thus completing the proof of \thmref{el ab lower bound theorem}.  We then prove our more delicate result, \thmref{general lower bound theorem}.

Section \ref{ex special 2-groups section} has the details of our results about extraspecial 2-groups, and has a proof of \thmref{good family thm}.

Section \ref{essential pairs section} has some final remarks regarding the general equivariant Balmer spectrum problem.  We discuss `exceptional pairs': pairs $(G,H)$ whose blue shift numbers cannot be deduced from smaller groups.  We also observe that the particular construction we use in our new 2-group examples -- $\RP(\omega)$ -- seems limited to improving the blue shift lower bound by at most 1.  The appendix has a table of exceptional pairs of 2-groups $(G,H)$ for $|G| \leq 32$.

\section{Background} \label{background section}

\subsection{$G$--spaces and classical Smith theory} \label{smith theory background sec}

The classical papers of both P.A.~ Smith and E.E.~Floyd predated notions like `$G$--C.W.~ complex' that we use in our introduction.  We say a bit about why their old theorems apply to $p$--local finite $G$--spaces as defined in Definition \ref{p finite defn}.

The first thing to say is that there are modern presentations in textbooks.

Chapter III of Bredon's 1972 book \cite{bredon trans groups} covers this material.  The theorem of Floyd (\ref{Floyd thm}) is \cite[Theorem III.4.1]{bredon trans groups}, the Euler characteristic result (\ref{Euler char thm}) is \cite[Theorem III.4.3]{bredon trans groups}, and the short deduction from these of Smith's theorem (\ref{Smith thm}) is given in \cite[\S III.5]{bredon trans groups}.  Bredon is working with finite dimensional regular $G$--complexes, but his proofs -- with arguments very similar to the original arguments of Smith and Floyd -- are working with equivariant cellular chains, and generalize without change to all finite dimensional $G$--C.W.~ complexes\footnote{A different work of Bredon \cite{Bredon eq co theories SLNM} from the same period is often given as the reference first defining $G$--C.W.~ complexes, but Bredon chose to not use these in \cite{bredon trans groups}}.

In section 4 of Chapter III of \cite[\S4]{tom Dieck}, tom Dieck gives the standard modern proof of this material using Borel equivariant cohomology and localization theorems: (\ref{Floyd thm}) is \cite[Proposition III.4.16]{tom Dieck}, and this is used to deduce (\ref{Smith thm}) (see \cite[Theorem III.4.22]{tom Dieck})\footnote{Curiously, a proof using localization seems to have first been given in yet another work of Bredon: \cite{Bredon 1967 proceedings}.}.  The Euler characteristic result (\ref{Euler char thm}) is \cite[Proposition III.6.7]{tom Dieck}.  His proofs apply to all finite dimensional $G$--C.W.~ complexes.

Now recall that, when $G$ is a finite $p$--group, we defined a based $G$--space to be a $p$--local finite $G$-space if it is a homotopy retract of the $p$--localization of a finite based $G$--C.W.~ complex.

\begin{lem} \label{nice lemma} If $X$ is a $p$--local finite $G$-space then inequality (\ref{Floyd thm}) holds.
\end{lem}
\begin{proof} Replacing $X$ by $\Sigma^2 X$ if needed, it suffices to prove the lemma when $X$ assumed to be the homotopy retract of the $p$--localization of a $G$--space $Y$ that is the double suspension of a finite based $G$--C.W.~ complex.  As already observed in \cite[paragraph after Lemma 2.2]{barthel greenlees hausmann}, the $p$--localization of such a $Y$ can be obtained as a mapping telescope of self maps of $Y$, and thus will be a finite dimensional $G$--C.W.~ complex (of one dimension more than $Y$).

Similarly the retract $X$ can be obtained as the mapping telescope of
$$ Y_{(p)} \xra{e} Y_{(p)} \xra{e} Y_{(p)} \xra{e} Y_{(p)} \xra{e} \dots$$
where $e$ is the composite $Y_{(p)} \twoheadrightarrow X \hra Y_{(p)}$, and thus will also be a finite dimensional $G$--C.W.~ complex (of one dimension more than $Y_{(p)}$). (Compare with \secref{idempotent section}.)

As $H_*(X;\Z/p)$ is a summand of $H_*(Y;\Z/p)$ and thus is clearly finite dimensional, Floyd's theorem applies.
\end{proof}

\begin{rem} We will use without comment easily verified facts like the following: if $X$ is a $p$--local finite $G$-space and $H$ is a normal subgroup of $G$, then $X^H$ is a $p$--local finite $G/H$-space.
\end{rem}

\subsection{The $p$--local equivariant stable category} \label{equi stable cat sec}

Given a finite $p$--group $G$, let $\Sp(G)$ be the category of $G$--spectra, equipped with an associative and commutative smash product, as in \cite{mandell may equivariant spectra}, and let $\Sp(G)_{(p)}$ the subcategory of $p$--local spectra.  Then let $\mathcal C_G$ denote the full subcategory of compact objects in the homotopy category $ho(\mathcal S(G)_{(p)})$.

We write $\Sp$ for $\Sp(\{e\})$ and $\mathcal C$ for  $\mathcal C_{\{e\}}$, so $\mathcal C$ is the category of compact objects in the homotopy category of (nonequivariant) $p$--local spectra $\Sp_{(p)}$.

As already stated in \lemref{unstable to stable lemma}, $Y \in \mathcal C_G$ if and only if it has the form $Y \simeq S^{-W}\Sinfty_GX$, where $W$ is a real representation of $G$, and $X$ is a $p$--local finite $G$--space.

The stable analogue of taking $H$--fixed points of $G$--spaces is the functor that assigns to a $G$--spectrum $Y$ its geometric $H$--fixed point spectrum $Y^{\Phi H}$.  This functor satisfies two basic properties:
\begin{itemize}
\item The functor $Y \mapsto Y^{\Phi H}$ is symmetric monoidal.
\item There are natural symmetric monoidal equivalences
$$ (\Sinfty_G X)^{\Phi H} \simeq \Sinfty (X^H).$$
\end{itemize}
From these properties, one can deduce that if $Y \simeq S^{-W} \sm \Sinfty_G X$, then $Y^{\Phi H} \simeq S^{-W^H} \sm \Sinfty (X^H)$, and so if $Y \in \mathcal C_G$, then $Y^{\Phi G} \in \mathcal C$.

\subsection{Morava $K$--theories} \label{morava k theory background section}

A general reference for this subsection is \cite{Wurgler 91}.

The coefficient ring of Morava $K$--theory is a graded field and $K(n)$ is a ring spectrum. These two facts imply that the natural Kunneth map
$$ \times: K(n)_*(X) \otimes_{K(n)_*}K(n)_*(Y) \ra K(n)_*(X \sm Y)$$
is an isomorphism for all spectra $X$ and $Y$, and that the natural duality map
$$ K(n)^*(Z) \ra \Hom_{K(n)_*}(K(n)_*(Z),K(n)_*)$$
is an isomorphism for all spectra $Z$.

At all primes, $K(n)$ is an associative ring spectra, and at odd primes it is also commutative.  This ensures that, when $p$ is odd, the functor
$$ K(n)_*(\text{\hspace{.1in}}): (\text{Spectra}, \sm, S) \ra (K(n)_*\text{--modules}, \otimes, K(n)_*)$$
is symmetric monoidal, where the symmetric monoidal stucture on $K(n)_*$--modules is the standard one: $V_* \otimes W_* = V_* \otimes_{K(n)_*} W_*$ with twist isomorphism
$\tau: V_* \otimes W_* \ra W_* \otimes V_*$ given by $\tau(x \otimes y) = (-1)^{|x||y|}y \otimes x$.

There is a wrinkle when $p=2$.  Let $t: X \sm Y \ra Y \sm X$ be the twist equivalence.  In \cite{Wurgler 86}, W\"urgler proves the formula
$$ t_*(x \times y) = y \times x + v_n(q(y) \times q(x)),$$
where $q: K(n)^*(X) \ra K(n)^{*+2^n-1}(X)$ is a natural derivation satisfying $q^2=0$.
It follows that $$ K(n)_*(\text{\hspace{.1in}}): (\text{Spectra}, \sm, S) \ra (K(n)_*\text{--modules}, \otimes, K(n)_*)$$
is {\em not} symmetric monoidal.  However, the formula suggests a fix, as follows.

Let $\Lambda_*(q)$ denote the graded Hopf algebra $K(n)_*[q]/(q^2)$, where $q$ is primitive and $|q|=2^n-1$.  Then we can regard the functor $X \mapsto K(n)_*(X)$ as taking values in $\Lambda_*(q)$--modules.  One easily checks that we can equip the category of $\Lambda_*(q)$--modules with an exotic symmetric monoidal structure:  given two such modules $U_*$ and $V_*$, let $U_* \otimes V_* = U_*\otimes_{K(n)_*}V_*$ as $\Lambda_*(q)$--modules, but with twist isomorphism $\tau: U_* \otimes V_* \ra V_* \otimes U_*$ given by $\tau(x \otimes y) = y \otimes x + v_n(q(y) \otimes q(x))$. W\"urgler's formula then tells us that
$$ K(n)_*(\text{\hspace{.1in}}): (\text{Spectra}, \sm, S) \ra (\Lambda_*(q)\text{--modules}, \otimes, K(n)_*)$$
{\em is} symmetric monoidal\footnote{The exotic symmetric monoidal category here has also appeared in algebraic settings: see the first pages of \cite{benson etingof}.}.

Finally, we remind readers of the fundamental Thick Subcategory Theorem of \cite{hs}:  If $\mathcal  B$ is a proper thick subcategory $\mathcal C$, then $\mathcal B = \mathcal C(n)$ for some $1\leq n\leq \infty$, where
$$\mathcal C(n) = \{ X \ | \ K(n-1)_*(X)=0\} = \{ X \ | \ X \text{ has type }\geq n\},$$
for finite $n$, and $\mathcal C(\infty) = \{*\}$.
With this notation, Mitchell and Ravenel's results combine to show that there are strict inclusions
$$ \mathcal C  \supset \mathcal C(1) \supset \mathcal C(2) \supset \dots, $$
and $\displaystyle \mathcal C(\infty) = \bigcap_{n<\infty} \mathcal C(n)$.
Each $\mathcal C(n)$ for $n \geq 1$ is an ideal in the tensor triangulated category $\mathcal C$, and, indeed, is even a prime ideal, thanks to the Kunneth theorem.

\section{Chromatic Smith theorems, the Balmer spectrum, and type functions} \label{balmer section}

The homotopy category of compact objects in $G$--equivariant spectra, with $G$ a finite group, is tensor triangulated, and in \cite{balmer sanders} the authors began the study of its Balmer spectrum.  They are able to reduce the study for all groups to the case of interest in this paper: understanding the Balmer spectrum of $\mathcal C_G$, the category of compact objects in $\mathcal S(G)_{(p)}$, when $G$ is a finite $p$--group.

By definition the points of the Balmer spectrum are the prime ideals in $\mathcal C_G$, and Balmer and Sanders check that these are precisely the prime ideals
$$ \Pp_{G}(H,n) = \{ X \in \mathcal C_G \ | \ X^{\Phi H} \in \mathcal C(n)\}, $$
with $H$ running through representatives of the conjugacy classes of subgroups of $G$,  and $n \geq 1$.

Understanding the topology of the Balmer spectrum is then shown to be equivalent to the following.

\begin{problem} \label{stable problem} Given a finite $p$--group $G$, and subgroups $K$ and $H$, for what pairs $(n,m)$ is it true that $\Pp_{G}(K,m) \subseteq \Pp_{G}(H,n)$?
\end{problem}
It isn't hard to show that a necessary condition for this to happen is that $K$ be subconjugate to $H$ in $G$.

The case when $G=H$ is related to Problem \ref{chro smith problems} as follows.

\begin{lem} \label{unstable = stable Smith lemma} Given $K\leq H$, the following are equivalent.

\noindent{(a)} $\Pp_{H}(K,m+1) \subseteq \Pp_{H}(H,n+1)$.

\noindent{(b)} For all $Y \in \mathcal C_H$, $K(m)_*(Y^{\Phi(K)}) = 0 \Rightarrow K(n)_*(Y^{\Phi(H)}) = 0$.

\noindent{(c)} The $(H,K,n,m)$ Chromatic Smith Theorem holds.

\noindent{(d)} $m \geq n + r_n(H,K)$.
\end{lem}
\begin{proof} It is clear from the definitions that (a) $\Leftrightarrow$ (b), and $r_n(H,K)$ was defined so that (c) $\Leftrightarrow$ (d).  Now recall that $Y \in \mathcal C_H$ if and only if $Y = S^{-W} \sm \Sinfty_H X$ with $X$ a $p$--local finite $H$--space, and, in this case, $Y^{\Phi(K)} = S^{-W^K} \sm \Sinfty X^K$.  Thus $K(m)_*(Y^{\Phi(K)}) = 0 \Leftrightarrow \widetilde K(m)_*(X^{K}) = 0$.  The equivalence (b) $\Leftrightarrow$ (c) follows, as (c) is the statement that for all $p$--local finite $H$--spaces $X$,  $\widetilde K(m)_*(Y^{K}) = 0 \Rightarrow \widetilde K(n)_*(Y^{H}) = 0$.
\end{proof}

Similarly, there is a Chromatic Floyd Theorem version of the equivalence (b) $\Leftrightarrow$ (c). It is useful to extend a bit of our notation:   if $Y$ is a spectrum such that $\dim_{K(n)_*} K(n)_*(Y)$ is finite, we let $k_n(Y) = \dim_{K(n)_*} K(n)_*(Y)$.  With this notation, $k_n(\Sinfty X) = k_n(X)$, for any based space such that $\dim_{K(n)_*} K(n)_*(X)$ is finite.

\begin{lem} \label{unstable = stable Floyd lemma} Given $K\leq H$, the following are equivalent.

\noindent{(a)} For all $Y \in \mathcal C_H$, $k_m(Y^{\Phi(K)}) \geq k_n(Y^{\Phi(H)})$.

\noindent{(b)} The $(H,K,n,m)$ Chromatic Floyd Theorem holds.
\end{lem}

Balmer and Sanders then show that Problem \ref{stable problem} for general $K,H \leq G$ reduces to the  $G=H$ case.  We describe how this goes.

It is useful to generalize our blue shift numbers\footnote{We note that $r^G_n(H,K)$ here is denoted $\beth_{n+1}(G;H,K)$ in \cite{6 author}.}.

\begin{defn}  Given $K<H<G$ and $n \geq 0$, let $r^G_n(H,K)$ be the minimal $r$ such that $\Pp_{G}(K,n+r+1) \subseteq \Pp_{G}(H,n+1)$.
\end{defn}

Thus $r^H_n(H,K) = r_n(H,K)$. \\

With this definition,  \cite[Prop.6.11]{balmer sanders} (when $G$ is a $p$--group) says the following.

\begin{prop} \label{G H K prop}  Given $K\leq H \leq G$,
$$r^G_n(H,K) = \min\{ r_n(H,L) \ | \ L \leq  H \text{ is conjugate to $K$ in $G$}\}.$$
\end{prop}

We end this section with a proof of this which is differently arranged than that in \cite{balmer sanders}. Enroute, we also give an elementary proof of a result in \cite{balmer sanders} about `type functions'.

Let $\text{Conj}(G)$ denote the set of conjugacy classes of subgroups of $G$.

\begin{defn}  Given a $p$--local finite $G$--space (or $p$--local compact $G$--spectrum) $X$, its {\em type function} is the function
$\text{type}_X: \text{Conj}(G) \ra \N \cup \{\infty\}$ defined by $\text{type}_X(H) = \text{ type }X^H$ (or $\text{ type }X^{\Phi(H)}$).
\end{defn}

The last statement in \cite[Corollary 10.6]{balmer sanders} tells us what type functions can occur.

\begin{prop} \label{type prop} Given a function $f: \text{Conj}(G) \ra \N \cup \{\infty\}$, there exists a $p$--local finite $G$--space $X$ such that $f=\text{type}_X$ if and only if $f(K) \leq f(H) + r_{f(H)}^G(H,K)$ for all $K<H<G$.
\end{prop}

To prove this, we start with a lemma.

\begin{lem} \label{type lemma} Given $K\lneq H<G$ and $n \geq 0$, there exists a $p$--local finite $G$--space $X$ such that $X^K$ has type $m$ and $X^H$ has type $n$ if and only if $m\leq n+ r^G_n(H,K)$.
\end{lem}
\begin{proof}  $m\leq n+ r^G_n(H,K)$ if and only if $\Pp_{G}(K,m) \nsubseteq \Pp_{G}(H,n+1)$, and this happens exactly when there exists a $p$--local finite $G$--space $Y$ with $\text{type }Y^K \geq m$ but $\text{type }Y^H \leq n$.  Given such a $Y$, let $X = (G/K_+ \sm U) \vee (Y \sm V)$ where $U$ has type $m$ and $V$ has type $n$, and both are given a trivial $G$-action.  Note that $(G/K)^K = W_G(K) (= N_G(H)/H)$, which is a nonempty finite set of points, while $(G/K)^H = \emptyset$.  Thus
$$ X^K = (W_G(K)_+ \sm U) \vee (Y^K \sm V)$$
which has type precisely $m$, while
$$ X^H = (\emptyset_+ \sm U) \vee (Y^H \sm V) = Y^H \sm V$$
which has type precisely $n$.
\end{proof}

\begin{proof}[Proof of \propref{type prop}]  The `only if' statement follows from the lemma.

For the `if' direction, suppose $f: \text{Conj}(G) \ra \N \cup \{\infty\}$ satisfies $f(K) \leq f(H) + r_{f(K)}^G(H,K)$ for all $K<H<G$. By the lemma, for each $K < H$, there exists a $p$--local finite $G$--space $Y(H,K)$ such that $\text{type }Y(H,K)^K = f(K)$ and $\text{type }Y(H,K)^H = f(H)$.

For each $H \in  \text{Conj}(G)$, we now let $X(H)$ be the $G$--space defined by
$$ X(H) = G/H_+ \sm \bigwedge_{L < H} Y(H,L).$$
We claim that $\text{type}_{X(H)}(K) \geq f(K)$ for all $K$, with equality when $K=H$. To see this, we first note if $K$ is not subconjugate to $H$, then $(G/H)^K = \emptyset$, so that $X(H)^K$ is contractible and $\text{type}_{X(H)}(K) = \infty$.  If $gKg^{-1} \lneq H$, then
\begin{equation*}
\begin{split}
\text{type }X(H)^K & = \text{type } \bigwedge_{L < H} Y(H,L)^K \\
  & = \max\{\text{type }Y(H,L)^K \ | \ L < H\} \\
  & \geq \text{type }Y(H, gKg^{-1})^K = f(K).
\end{split}
\end{equation*}
Finally, if $K=H$, then $\displaystyle \text{type }X(H)^H = \text{type } \bigwedge_{L < H} Y(H,L)^H = f(H)$.

Now let $\displaystyle X = \bigvee_H X(H)$.  Then $\displaystyle \text{type}_X(K) = \min_H\{ \text{type}_{X(H)}(K)\} = f(K)$.
\end{proof}

We now turn to the proof of \propref{G H K prop}.  It is useful to write $L \overset{G}{\sim}K$ if $L$ and $K$ are conjugate subgroups of $G$.

\begin{lem} \label{induction type lemma}  Given $K\leq H<G$, and a $p$--local finite $H$--space $X$, the $p$--local finite $G$--space $G_+ \sm_H X$ satisfies:
$$ \text{type } (G_+ \sm_H X)^H = \text{ type } X^H,$$
and
$$ \text{type } (G_+ \sm_H X)^K = \min\{\text{type } X^L \ | \ L \leq H \text{ and } L \overset{G}{\sim}K\}.$$
\end{lem}
\begin{proof}  The first statement is a special case of the second. For the second, one checks that
$$ (G_+ \sm_H X)^K = \bigvee_{gH \in (G/H)^K} X^{g^{-1}Kg}$$
and that $gH \in (G/H)^K$ if and only if $g^{-1}Kg \leq H$.
\end{proof}

\begin{proof}[Proof of \propref{G H K prop}]   Recall that our goal is to show that if $K<H<G$, then
$r^G_n(H,K) = \min\{ r_n(H,L) \ | \ L < H \text{ and } L \overset{G}{\sim}K\}$.
\lemref{type lemma} allows us to regard this as a statement about type functions.

We first check that $r^G_n(H,K) \leq \min\{ r_n(H,L) \ | \ L < H \text{ and } L \overset{G}{\sim}K\}$.  To see this, let $Y$ be a $p$--local finite $G$--space such that $Y^H$ has type $n$ and $Y^K$ has type $n+r^G_n(H,K)$.  Suppose $L = g^{-1}Kg < H$. If we consider $Y$ as an $H$--space by restriction, then $Y^H$ has type $n$ and $Y^L$ still has type $n+r^G_n(H,K)$ since $Y^L = g^{-1}Y^K$.  Thus $r_n(H,L) \geq r_n^G(H,K)$.

We show the other inequality holds.  Given $J<H$ with $J \overset{G}{\sim}K$, let $X(J)$ be a $p$--local finite $H$--space such that $X(J)^H$ has type $n$ and $X(J)^J$ has type $n+r_n(H,J)$, and  let $\displaystyle X = \bigwedge_{J} X(J)$.  Then $\displaystyle X^H = \bigwedge_{J} X(J)^H$, which still has type $n$, while, if $L<H$ then $\displaystyle X^L = \bigwedge_{J} X(J)^L$, so that if also $L \overset{G}{\sim}K$ then
$$ \text{type }X^L = \max\{\text{type }X(J)^L \ | \ J \overset{G}{\sim}K\}\geq \text{ type }X(L)^L = n+r_n(H,L).$$
Applying \lemref{induction type lemma} to the $G$--space $Y = G_+ \sm_H X$, we see that $Y^H$ has type $n$, while
\begin{equation*}
\begin{split}
\text{type }Y^K &
= \min\{\text{type } X^L \ | \ L < H \text{ and } L \overset{G}{\sim}K\} \\
  & \geq \min\{n+ r_n(H,L) \ | \ L < H \text{ and } L \overset{G}{\sim}K\}.
\end{split}
\end{equation*}
This means that $r^G_n(H,K) \geq \min\{ r_n(H,L) \ | \ L < H \text{ and } L \overset{G}{\sim}K\}$.
\end{proof}

The following is likely the simplest example illustrating the difference between $r_n(H,K)= r_n^H(H,K)$ and $r^G_n(H,K)$.  (We thank Richard Lyons for pointing us towards this.)

\begin{ex} \label{Lyons' example}  Let $G = (C_4 \times C_4)\rtimes C_2$, $H = C_4 \times C_2 < C_4 \times C_4 < G$, $K = C_2 \times \{e\} <H$, and $L = \{e\} \times C_2 <H$.  Then $L \overset{G}{\sim}K$, and, by the abelian group results of \cite{6 author}, $r_n(H,K)= \text{rank }H/K = 2$, while $r_n(H,L)= \text{rank }H/L = 1$.  It follows that $r_n^G(H,K) = 1$.
\end{ex}

\section{Basic properties of $r_{-}(G,H)$, $r_n(G,H)$, and $r_{+}(G,H)$.} \label{r numbers section}

Here we discuss some basic properties of the blue shift numbers $r_n(G,H)$ and their group theoretic upper and lower bounds, $r_{+}(G,H)$ and $r_{-}(G,H)$, and how these bounds are deduced from the following two results from \cite{6 author}.

\begin{thm} \label{cyclic group upper bound} $r_n(C_{p^k},\{e\}) \leq 1$ for all $n$ and $k$.
\end{thm}

This upper bound is \cite[Thm.2.1]{6 author}, specialized to the case when $A$ is cyclic.  (Though not noted in \cite{6 author}, their result for a general finite abelian $p$--group $A$ follows from the cyclic group case.)  The proof is via localization theory, and thus, in some sense, resembles the proof of P.A.Smith's theorem as in \cite{tom Dieck}.  For an alternate proof of \thmref{cyclic group upper bound}, see \cite{kuhn 21}.

We restate \thmref{el ab lower bound theorem}.

\begin{thm} \label{el ab lower bound theorem again} $r \leq r_n(C_{p}^r,\{e\})$ for all $n$ and $r$.
\end{thm}

This is \cite[Thm.2.2]{6 author}, and in \secref{examples section} we will give the details of our simpler proof of this lower bound, following the outline in \secref{old examples intro subsection}.

Recall that if $H$ is a subgroup of a finite $p$--group $G$, $r_{+}(G,H)$ is defined to be the minimal $r$ such that there exists a chain of subgroups $ H = K_0 \lhd K_1 \lhd \dots \lhd K_r = G$ with each $K_{i-1}$ normal in $K_i$ and  $K_i/K_{i-1}$ cyclic.

We also defined $r_{-}(G,H)$ to be the rank of $G/H\Phi(G)$. One easily sees that $r_{-}(G,H) = r_+(G,H\Phi(G))$.

The following property is elementary but very useful.

\begin{lem}  If $N$ is normal in $G$, then $r_{+}(G,H) \geq r_{+}(G/N,HN/N)$ and $r_{-}(G,H) \geq r_{-}(G/N,HN/N)$, with equality in both cases if $N \leq H$.
\end{lem}
\begin{proof} For $r_{+}$, the image in $G/N$ of a minimal subgroup chain between $H$ and $G$ with cyclic subquotients will be a chain between $HN/N$ and $G/N$ with cyclic subquotients.  If $N \leq H$ this will be a bijection between such chains.  The statement for $r_{-}$ can be easily checked directly, or deduced from the $r_{+}$ case, since $r_{-}(G,H) = r_{+}(G,H\Phi(H))$.
\end{proof}

\begin{cor} $r_{+}(G,H) \geq r_{-}(G,H)$.
\end{cor}
\begin{proof} Specializing the lemma to the case when $N = H\Phi(G)$, one learns that
$r_{+}(G,H) \geq r_{+}(G/H\Phi(G),\{e\}) = r_{+}(G,H\Phi(H)) = r_{-}(G,H)$.
\end{proof}

The analogue of the last lemma also holds for $r_{n}(G,H)$.

\begin{lem} If $N$ is normal in $G$, then $r_{n}(G,H) \geq r_{n}(G/N,HN/N)$ for all $n$, with equality if $N \leq H$.
\end{lem}
\begin{proof} \lemref{type lemma} tells us that there exists a $p$--local finite $G/N$--space $X$ such that $X^{G/N}$ has type $n$ and $X^{HN/N}$ has type $n+r_n(G/N,HN/N)$.  If we regard $X$ as a $G$--space via the quotient map $G \ra G/N$, then $X^G = X^{G/N}$ has type $n$ and $X^H = X^{HN/N}$ has type $n+r_n(G/N,HN/N)$. It follows that $n+r_n(G/N,HN/N) \leq n+ r_{n}(G,H)$.

Now suppose that $N \leq H$, and that $Y$ is a $p$--local finite $G$--space such that $Y^G$ has type $n$ and $Y^H$ has type $n+r_n(G,H)$.  If we let $X = Y^N$, then $X$ will be a $p$--local finite $G/N$--space such that $X^{G/N} = Y^G$ has type $n$ and $X^{HN/N} = Y^H$ has type $n+r_n(G,H)$.  Thus $n+ r_{n}(G,H) \leq n+r_n(G/N,HN/N)$.
\end{proof}

Specializing this lemma to the case when $N = H\Phi(G)$, one learns that $r_{n}(G,H) \geq r_{n}(G/H\Phi(G),\{e\})$, which \thmref{el ab lower bound theorem again} tells us is at least a big as the rank of $G/H\Phi(G)$.  We learn the following.

\begin{cor} $r_{n}(G,H) \geq r_{-}(G,H)$ for all $n$. \label{lower bound cor}
\end{cor}

Another useful corollary goes as follows.

\begin{cor} \label{normal cor} If $N$ is normal in $G$, then $r_n(G,N) = r_n(G/N, \{e\})$ for all $n$.
\end{cor}

Now we note some transitivity properties.

\begin{lem} \label{trans lemma} Let $H<K<G$. \\

\noindent{\bf (a)} \ $r_{+}(G,H) \leq r_{+}(G,K)+ r_{+}(K,H)$. \\

\noindent{\bf (b)} \ $r_{-}(G,H) \leq r_{-}(G,K)+ r_{-}(K,H)$.
\end{lem}

Both of these inequalities are clear from the definitions. Strict inequality can certainly hold in both cases: consider $\{e\} < C_2 < C_4$.

An analogous property for the blue shift numbers goes as follows.

\begin{lem} Let $H<K<G$. Then
$$ r_{n}(G,H) \leq r_{n}(G,K)+ \max_{m\leq n+r_n(G,K)}\{r_{m}(K,H)\}.$$
\end{lem}
\begin{proof}  Let $l = n+r_n(G,H)$, and let $X$ be a $p$--local finite $G$--space such that $X^G$ has type $n$ and $X^H$ has type $l$.  Let $m$ be the type of $X^K$.  Then $m \leq n+r_n(G,K)$ and $l \leq m + r_m(K,H)$, so that
$$r_{n}(G,H)  = l-n = (m-n) + (l-m) \leq r_{n}(G,K)+ r_{m}(K,H).$$
\end{proof}

\begin{cor} $r_n(G,H) \leq r_+(G,H)$ for all $n$. \label{upper bound cor}
\end{cor}
\begin{proof} We prove this by induction on $r_+(G,H)$.

$r_+(G,H) = 1$ means that $H$ is normal in $G$ and $G/H$ is cyclic.  But then $r_n(G,H) = r_n(G/H, \{e\})$ by \corref{normal cor}, and $r_n(G/H, \{e\}) \leq 1 = r_+(G,H)$ by \thmref{cyclic group upper bound}.

For the inductive step, let $ H = K_0 \lhd K_1 \lhd \dots \lhd K_r = G$ be a minimal chain
with each $K_i/K_{i-1}$ cyclic.  By inductive hypothesis, $r_n(G,K_1) \leq r_+(G,K_1)$ for all $n$, and $r_{m}(K_1,H) \leq 1$ for all $m$, by the case just discussed. The lemma applied to $H<K_1<G$ then implies that
\begin{equation*}
\begin{split}
r_{n}(G,H) &
\leq r_{n}(G,K_1)+ \max_{m\leq n+r_n(G,K_1)}\{r_{m}(K_1,H)\}  \\
  & \leq r_+(G,K_1) + 1 = r_+(G,H).
\end{split}
\end{equation*}
\end{proof}

\corref{lower bound cor} and \corref{upper bound cor} combine to show that
$$ r_{-}(G,H) \leq  r_n(G,H) \leq r_{+}(G,H)$$
for all $H<G$ and all $n$.

We end this section by noting some differences between our upper and lower bound functions.

\begin{lem} If $H<K<G$, then $r_{-}(G,H) \geq r_{-}(G,K)$.
\end{lem}

\begin{ex} Let $D_{16}$ be the dihedral group of order 16, and let $C$ be a noncentral subgroup of order 2.  Then $r_+(D_{16},\{e\}) = 2$, while $r_+(D_{16},C) = 3$. (We note that $r_-(D_{16},\{e\}) = 2$, while $r_-(D_{16},C) = 1$.)
\end{ex}

\begin{ex} Related to this last example, let $D_{2^{k+1}}$ be the dihedral group of order $2^{k+1}$ and let $C$ be a noncentral subgroup of order 2.  Then $r_-(D_{2^{k+1}},C) = 1$ while $r_+(D_{2^{k+1}},C) = k$.  This illustrates that $r_+(G,H) - r_-(G,H)$ can be arbitrarily large.
\end{ex}

\begin{lem}  If $H_1 < G_1$ and $H_2 < G_2$, then
$$r_{-}(G_1 \times G_2, H_1 \times H_2) = r_{-}(G_1,H_1) + r_{-}(G_2,H_2).$$
\end{lem}

Regarding our upper bound, \lemref{trans lemma}(a) implies that
$$r_{+}(G_1 \times G_2, H_1 \times H_2) \leq r_{+}(G_1,H_1) + r_{+}(G_2,H_2),$$
but the next example shows that strict inequality can happen, even when one of the pairs is as trivial as possible.

\begin{ex} \label{M_4(2) product example} Let $a$ be the generator of the cyclic group $C_2$, and let $M_4(2)$ be the `modular maximal--cyclic group of order 16', a group generated by elements $b,c$ satisfying $b^8 = c^2 = e$ and $cbc = b^5$.  (It is group 16\#6 in the GAP Small Groups Library.) Let $C = \langle c \rangle$, a noncentral subgroup of order 2.   Then $r_+(C_2,\{e\}) = 1$ and $r_+(M_4(2),C) = 2$, but $r_{+}(C_2 \times M_4(2), \{e\} \times C) = 2$.  To see this last fact, we have a chain of normal subgroups
$$ \{e\} \times C = \langle c \rangle \lhd \langle ab^2, c \rangle \lhd \langle a,b,c \rangle = C_2 \times M_4(2),$$
with $\langle ab^2, c \rangle/\langle c \rangle \simeq \langle ab^2 \rangle \simeq C_4$, and
$\langle a,b,c \rangle/\langle ab^2, c \rangle \simeq \langle a,b \rangle/\langle ab^2 \rangle \simeq C_4$.
\end{ex}

\begin{rem}  From this one learns that $r_n(C_2 \times M_4(2), \{e\} \times C) = 2$ for all $n$, while $r_n(M_4(2),C)$ can't be determined: it is either $1=r_-(M_4(2),C)$ or $2=r_+(M_4(2),C)$.
\end{rem}

\section{Jeff Smith's construction, and the proof of \thmref{equivalence thm}} \label{proof of smith = floyd thm}

The goal of this section is to prove \thmref{equivalence thm} which said that if the $(G,H,n,m)$ Chromatic Smith Theorem is true then the $(G,H,n,m)$ Chromatic Floyd Theorem is true.

In fact, we prove the contrapositive of this, and thanks to \lemref{unstable = stable Smith lemma} and \lemref{unstable = stable Floyd lemma}, we can translate this into the following statement about spectra in $\mathcal C_G$:

{\em Suppose that there exists $X \in \mathcal C_G$ such that $k_m(X^{\Phi(H)})< k_n(X^{\Phi(G)})$.
Then there exists $F \in \mathcal C_G$ with $k_m(F^{\Phi(H)})=0$ and $k_n(F^{\Phi(G)}) \neq 0$.} \\

This then will follow from the following more general theorem.

\begin{thm} \label{smash summand theorem}  Suppose $X \in \mathcal S(G)_{(p)}$ satisfies
$$\dim_{K(m)_*} K(m)_*(X^{\Phi(H)})< \dim_{K(n)_*} K(n)_*(X^{\Phi(G)})$$
for some $H<G$, $m$, and $n$. Then there exists a $k$, and an idempotent $e \in \Z_{(p)}[\Sigma_k]$ such that the associated wedge summand $F=eX^{\sm k}$ of $X^{\sm k}$ satisfies  $K(m)_*(F^{\Phi(H)})=0$ while $K(n)_*(F^{\Phi(G)}) \neq 0$.
\end{thm}

In the remainder of this section, we explain what this means and prove the theorem.

\subsection{Using idempotents to split $G$--spectra} \label{idempotent section}  Given $Y \in \Sp(G)$, $\{Y,Y\}_G$, the set of homotopy classes of self maps of $Y$, is a ring with multiplication given by composition\footnote{Our reason for working with spectra in this section is precisely because rings of homotopy classes of self maps form a ring, which is not necessarily the case for based spaces, even those which are double suspensions.}.

Recall that $e \in \{Y,Y\}_G$ is idempotent if $e^2=e$.  The map $e$ then determines a wedge summand $eY$ of $Y$, or, more precisely, an object $eY \in \Sp(G)$, together with maps $r: Y \ra eY$ and $i: eY \ra Y$ such that the diagram
\begin{equation} \label{idempotent diagram}
\SelectTips{cm}{}
\xymatrix{
 Y \ar[dr]^r \ar[rr]^e && Y \ar[dr]^r &  \\
& eY \ar@{=}[rr] \ar[ur]^i && eY }
\end{equation}
commutes up to homotopy.

There are two well known ways to show this.

The spectrum $eY$ can be defined to be the mapping telescope
$$ Tel \{Y \xra{e} Y  \xra{e}Y \xra{e} Y \xra{e} \dots\},$$
and properties of this construction define $r$ and $i$.

Alternatively, and equivalently, $eY$ can be defined using Brown Representability: $eY \in \Sp(G)$ is defined so that there is a natural isomorphism $\{Z,eY\}_G \simeq e^*\{Z,Y\}_G$ for all $Z$. The maps $r$ and $i$ then correspond to evident natural transformations.

The first approach is used in \cite[\S 6.4]{ravenel orange book}, the second is described in \cite[\S 2]{harris kuhn}, both are mentioned in \cite[\S 1]{mitchell math zeit}, and these constructions likely go back to the 1960's.

This construction behaves well with respect to taking geometric fixed points.  The idempotent $e: Y \ra Y$ will induce an idempotent map on $Y^{\Phi(H)}$ for each $H \leq G$; by abuse of notation, we write this as $e: Y^{\Phi(H)} \ra Y^{\Phi(H)}$.  Applying the geometric fixed point functor for $H$ to (\ref{idempotent diagram}) yields the homotopy commutative diagram
\begin{equation}
\SelectTips{cm}{}
\xymatrix{
 Y^{\Phi(H)} \ar[dr]^r \ar[rr]^e && Y^{\Phi(H)} \ar[dr]^r &  \\
& (eY)^{\Phi(H)} \ar@{=}[rr] \ar[ur]^i && (eY)^{\Phi(H)}, }
\end{equation}
from which one can conclude that there is a natural equivalence
$$(eY)^{\Phi(H)} \simeq e(Y^{\Phi(H)}).$$

Relevant for us, note that, for all $n\geq0$, there will thus be a natural isomorphism
$$K(n)_*((eY)^{\Phi(H)}) \simeq eK(n)_*(Y^{\Phi(H)}),$$
where we write $e:K(n)_*(Y^{\Phi(H)}) \ra K(n)_*(Y^{\Phi(H)})$ for the homomorphism induced by $e: Y^{\Phi(H)} \ra Y^{\Phi(H)}$.

\subsection{Using idempotents in the group ring of symmetric groups}

We now specialize the discussion of the last subsection to the case when $Y = X^{\sm k}$, with $X \in \Sp(G)_{(p)}$.

The $k$th symmetric group $\Sigma_k$ acts naturally on the $k$--fold smash product $X^{\sm k}$ by permuting the factors, inducing a ring homomorphism
$$\Z[\Sigma_k] \ra \{X^{\sm k}, X^{\sm k}\}_G.$$
Since $X$ is $p$-local, this extends to a ring homomorphism
$$\Z_{(p)}[\Sigma_k] \ra \{X^{\sm k}, X^{\sm k}\}_G.$$
Thus an idempotent $e \in \Z_{(p)}[\Sigma_k]$ induces an idempotent map $e: X^{\sm k} \ra X^{\sm k}$, and thus a wedge summand $e X^{\sm k} \in \Sp(G)_{(p)}$.

There are  then natural equivalences
$$ (e X^{\sm k})^{\Phi(H)} \simeq e((X^{\sm k})^{\Phi(H)}) \simeq e ((X^{\Phi(H)})^{\sm k}).$$
Here, the second equivalence holds because the geometric fixed point functors are monoidal.

From this, and the Kunneth isomorphism for Morava $K$--theory, we see that there will be natural isomorphisms, for all $n$,
\begin{equation} \label{Kn(eX smash k) equation} K(n)_*((e X^{\sm k})^{\Phi(H)})\simeq e (K(n)_*(X^{\Phi(H)})^{\otimes k}),
\end{equation}
where the action of $\Sigma_k$ on $K(n)_*(X^{\Phi(H)})^{\otimes k}$ is the standard (signed) permutation action
if $p$ is odd (or $n=0$), and is the action induced by the exotic symmetric monoidal structure described in \secref{morava k theory background section} if $p=2$.

\subsection{The strategy of the proof of \thmref{smash summand theorem}}

Now we can explain the strategy of the proof of \thmref{smash summand theorem}.

Our assumption is that $X \in \mathcal S(G)_{(p)}$ satisfies
$$\dim_{K(m)_*} K(m)_*(X^{\Phi(H)})< \dim_{K(n)_*} K(n)_*(X^{\Phi(G)}).$$
(We include the possibility that $\dim_{K(n)_*} K(n)_*(X^{\Phi(G)})$ is infinite.)

It is convenient to let $W_* = K(m)_*(X^{\Phi(H)})$ and let $V_* = K(n)_*(X^{\Phi(G)})$, so that
$$ \dim_{K(m)_*} W_* < \dim_{K(n)_*} V_*.$$

Given $k \geq 1$ and an idempotent $e \in \Z_{(p)}[\Sigma_k]$, let $F = eX^{\sm k}$.  Then (\ref{Kn(eX smash k) equation}) tells us that
$$ K(m)_*(F^{\Phi(H)})\simeq e W_*^{\otimes k},$$
while
$$ K(n)_*(F^{\Phi(G)})\simeq e V_*^{\otimes k}.$$

Now suppose that one can always choose $k$ and an idempotent $e \in \Z_{(p)}[\Sigma_k]$ such that if $\F_*$ is any graded field\footnote{$\F_*$ will automatically be concentrated in even degrees if $p\neq 2$.} of characteristic 0 or $p$, and $U_*$ is a graded $\F_*$--vector space, then
\begin{equation} \label{`almost' prop}
\dim_{\F_*} eU_*^{\otimes k} \
\begin{cases}
\ = 0 & \text{if } \dim_{\F_*} U_* \text{ is `small'} \\  \  \neq 0 & \text{if } \dim_{\F_*} U_* \text{ is `large'},
\end{cases}
\end{equation}
where `small' includes $ \dim_{K(m)_*} W_*$ and `large' includes $ \dim_{K(n)_*} V_*$.
Then we could conclude that $ K(m)_*(F^{\Phi(H)})= 0$, while $ K(n)_*(F^{\Phi(G)}) \neq 0$, proving \thmref{smash summand theorem}.  (We need to include the possibility that $\F_*$ has characteristic 0 to deal with the cases when $m=0$ or $n=0$ in \thmref{smash summand theorem}).

We will `almost' show that this can be done, but there are couple of issues here:
\begin{itemize}
\item  When $p\neq 2$, one needs to handle $U_*$'s that have both even and odd dimensional parts.
\item  When $p=2$, one needs to handle the exotic symmetric monoidal structure.
\end{itemize}

The first of these is by far the most significant: one needs to keep track of the dimensions of the even and odd dimensional parts of our graded vector spaces.  In the next subsection, we will describe the precise version of (\ref{`almost' prop}) that we will use to complete the proof of \thmref{smash summand theorem}.

The second issue is easily `filtered away', as we now explain.

When $p=2$, recall that $K(n)_*(X)$ is naturally a $\Lambda_*(q)$--module where $\Lambda_*(q)$ denotes the graded Hopf algebra $K(n)_*[q]/(q^2)$, with $|q|=2^n-1$.

Our exotic structure has $U_* \otimes V_* = U_* \otimes_{K(n)_*} V_*$ as $\Lambda_*(q)$--modules, with twist isomorphism
$\tau: U_* \otimes V_* \ra V_* \otimes U_*$ given by
$\tau(x \otimes y) = y \otimes x + v_n(q(y) \otimes q(x))$.

Let $\otimes_u$ denote the standard structure: $U_* \otimes_u V_*$ is still $U_* \otimes_{K(n)_*} V_*$ as $\Lambda_*(q)$--modules, but now with twist isomorphism $\tau: U_* \otimes_u V_* \ra V_* \otimes_u U_*$ given by $\tau(x \otimes y) = y \otimes x $.

\begin{lem}  \label{p=2 twist lemma}  Let $U_*$ be a $\Lambda_*(q)$--module that is finite dimensional as a $K(n)_*$--vector space. For any $k \geq 1$ and idempotent $e \in \Z_{(2)}[\Sigma_k]$,
$$\dim_{K(n)_*} e U_*^{\otimes k} = \dim_{K(n)_*} e U_*^{\otimes_u k}.$$
\end{lem}

Thus, for the purposes of proving our theorem, we can replace the exotic twist isomorphism with the standard one.
\begin{proof}[Proof of \lemref{p=2 twist lemma}]
We filter $U_*$ as a $\Lambda_*(q)$--module in the simplest way possible: let $F_0 = \ker q$, and then let $F_1= U_*$.  This induces a filtration on $U_*^{\otimes k}$ as a $\Lambda_*(q)[\Sigma_k]$--module, with associated graded module $(F_0 \oplus F_1/F_0)^{\otimes k}$.  But this is just $(F_0 \oplus F_1/F_0)^{\otimes_u k}$, since $q$ acts trivially on $(F_0\oplus F_1/F_0)$.  Applying an idempotent is an exact process, and thus idempotents commute with filtrations. Thus we have
\begin{equation*}
\begin{split}
\dim_{K(n)_*} e U_*^{\otimes k} &
= \dim_{K(n)_*} e (F_0\oplus F_1/F_0)^{\otimes k} \\
  & = \dim_{K(n)_*} e (F_0\oplus F_1/F_0)^{ \otimes_u k} = \dim_{K(n)_*} e U_*^{\otimes_u k}.
\end{split}
\end{equation*}
\end{proof}

\subsection{Properties of some idempotents and the proof of \thmref{smash summand theorem}}

If $\F_*$ is a graded field, and $V_*$ is a graded $\F_*$--vector space, $V_*$ will have a canonical decomposition $V_* = V_*^e \oplus V_*^o$ into its even and odd graded parts.

We are considering $V_*^{\otimes k}= (V_*^e \oplus V_*^o)^{\otimes k}$, the $k$--fold tensor product over $\F_*$, of $V_*$ with itself, viewed as a $\F_*[\Sigma_k]$--module with the usual sign conventions.

\begin{defn} Let $k(d) = \binom{d+1}{2}$ for $d \geq 1$.
\end{defn}

The next two propositions describe properties of some idempotents that will suffice to complete the proof of \thmref{smash summand theorem}.

\begin{prop} \label{p=2 prop}  For all $d$, there exists an idempotent $e_d \in \Z_{(2)}[\Sigma_{k(d)}]$ such that, for all graded fields $\F_*$ of characteristic 0 or 2, and all finite dimensional graded $\F_*$--vector spaces $V_*$,
$$e_d V_*^{\otimes k(d)} \neq 0 \text{ if and only if } \dim_{\F_*} V_* \geq d.$$
\end{prop}

\begin{prop} \label{p odd prop}  Let $p$ be odd. For all $d$, there exist idempotents $e_d, e_d^{\prime} \in \Z_{(p)}[\Sigma_{(p-1)k(d)}]$ such that, for all graded fields $\F_*$ of characteristic 0 or p, and all finite dimensional graded $\F_*$--vector spaces $V_*$,
$$e_d V_*^{\otimes (p-1)k(d)} \neq 0 \text{ if and only if } (p-1)\dim_{\F_*} V_*^e + \dim_{\F_*} V_*^o \geq (p-1)d.$$
and
$$e_d^{\prime} V_*^{\otimes (p-1)k(d)} \neq 0 \text{ if and only if } (p-1)\dim_{\F_*} V_*^o + \dim_{\F_*} V_*^e \geq (p-1)d.$$
\end{prop}

We postpone a discussion of the proofs to the next subsection.  Assuming these propositions, we now complete the proof of \thmref{smash summand theorem}.

We will use the following elementary lemma in the case when $p$ is odd.

\begin{lem}  Let $p$ be an odd prime.  Suppose that nonnegative integers $a_e$, $a_o$, $b_e$, and $b_o$ satisfy $a_e+a_o < b_e+b_o$.  Then there exists $d$ such that at least one of the following inequalities holds:
$$ (p-1)a_e + a_o < (p-1)d \leq (p-1)b_e + b_o,$$
$$ (p-1)a_o + a_e < (p-1)d \leq (p-1)b_o + b_e.$$
\end{lem}
\begin{proof}  The inequality $a_e+a_o < b_e+b_o$ implies that $(b_e-a_e) + (b_o-a_o) \geq 1$, so at least one of the inequalites $(b_e-a_e)\geq 1$ and $(b_o-a_o)\geq 1$ is true.

If $(b_e-a_e)\geq 1$, then $(p-2)(b_e-a_e)\geq p-2$. Adding that to $(b_e-a_e) + (b_o-a_o) \geq 1$ shows that $[(p-1)b_e + b_o] - [(p-1)a_e + a_o] \geq p-1$.  This, in turn, clearly implies that there exists $d$ such that
$$ (p-1)a_e + a_o < (p-1)d \leq (p-1)b_e + b_o.$$

The case when $(b_o-a_o)\geq 1$ is similar.
\end{proof}

\begin{proof}[Proof of \thmref{smash summand theorem}]  Suppose that a $p$--local $G$--spectrum $X$ satisfies $$\dim_{K(m)_*} K(m)_*(X^{\Phi(H)})< \dim_{K(n)_*} K(n)_*(X^{\Phi(G)}).$$
We need to show that there is then a finite $G$--spectrum $F$ of the form $F = eX^{\sm k}$ with $K(m)_*(F^{\Phi(H)})=0$ and $K(n)_*(F^{\Phi(G)}) \neq 0$.

If $p=2$, we let $d = \dim_{K(m)_*} K(m)_*(X^{\Phi(H)})+1$, so that
$$ \dim_{K(m)_*} K(m)_*(X^{\Phi(H)}) < d \leq \dim_{K(n)_*} K(n)_*(X^{\Phi(G)}).$$
Now  let $F = e_d X^{\sm k(d)}$.  If we let $W_* = K(m)_*(X^{\Phi(H)})$ and $V_* = K(n)_*(X^{\Phi(G)})$,
then $K(m)_*(F^{\Phi(H)}) \simeq e_dW_*^{\otimes k(d)}$ and $K(n)_*(F^{\Phi(G)}) \simeq e_dV_*^{\otimes k(d)}$.  \propref{p=2 prop} then applies to tell us that $F$ has the desired properties:
$$K(m)_*(F^{\Phi(H)})=0 \text{ and } K(n)_*(F^{\Phi(G)}) \neq 0.$$

If $p$ is odd, one argues similarly, using \propref{p odd prop}. The lemma shows that there exists a $d$ such that at least one of the following is true:
\begin{multline*}\tag*{{\bf (a)}}
(p-1)\dim_{K(m)_*} K(m)_*(X^{\Phi(H)})^e+ \dim_{K(m)_*} K(m)_*(X^{\Phi(H)})^o \\
< (p-1)d \leq (p-1)\dim_{K(n)_*} K(n)_*(X^{\Phi(G)})^e+ \dim_{K(n)_*} K(n)_*(X^{\Phi(G)})^o,
\end{multline*}
\begin{multline*}\tag*{{\bf (b)}}
(p-1)\dim_{K(m)_*} K(m)_*(X^{\Phi(H)})^o+ \dim_{K(m)_*} K(m)_*(X^{\Phi(H)})^e \\
< (p-1)d \leq (p-1)\dim_{K(n)_*} K(n)_*(X^{\Phi(G)})^o+ \dim_{K(n)_*} K(n)_*(X^{\Phi(G)})^e.
\end{multline*}
If (a) holds, one lets $F = e_d X^{\sm (p-1)k(d)}$.  If (b) holds, one lets $F = e_d^{\prime} X^{\sm (p-1)k(d)}$.
\end{proof}

\begin{rem} It was an insight of Jeff Smith in the mid 1980's that the classic representation theory literature offered formulae for idempotents that would have properties like those in the propositions, and that one could make good use of these, with arguments similar to those here\footnote{His application was to construct an explicit finite complex whose mod $p$ cohomology as an $A$--module made it relatively easy to show it was type $n$, and also, using standard Adams spectral sequence methods, that it admitted a $v_n$--self map: see \cite{hs}.}. His work with idempotents appears in a manuscript \cite{jeff smith} dating from around 1990, but is, sadly, unpublished. However both background and the results needed for the applications in \cite{hs}, and also by us here, appear in \cite[Appendix C]{ravenel orange book}.
\end{rem}

\subsection{The special idempotents}

In this section we describe idempotents $e_d$ and $e^{\prime}_d$ with the properties listed in Propositions \ref{p=2 prop} and \ref{p odd prop}.

To give the reader a sense of how this might go, we first offer two examples, which describe the first interesting cases.

\begin{ex}  When $p$ is odd and $d=1$, one can define $e_1, e^{\prime}_1 \in \Z_{(p)}[\Sigma_{p-1}]$ as follows:
$$ e_1 = \frac{1}{(p-1)!} \sum_{\sigma \in \Sigma_{p-1}} [\sigma] \text{ and } e^{\prime}_1 = \frac{1}{(p-1)!} \sum_{\sigma \in \Sigma_{p-1}} \text{sgn}(\sigma)[\sigma].$$
It is not hard to see that $e_1 (V_*)^{\otimes p-1} \simeq \bigoplus_{i+j=p-1} \Gamma^i(V_*^e)\otimes \Lambda^j(V_*^o)$, where $\Gamma^i$ is the $i$th symmetric invariants functor, and $\Lambda^j$ is the $j$th exterior power functor. It follows that $e_1V_*^{\otimes p-1} = 0$ if and only if $V_*^e=0$ and $\dim_{\F_*}V_*^o < p-1$.

For $e_1^{\prime}$, the roles of $V^e_*$ and $V^o_*$ are reversed.
\end{ex}

\begin{ex}  When $p=2$, $e_2 \in \Z_{(2)}[\Sigma_3]$ is the first interesting case.  There are transpositions $s = (12)$ and $t = (13)$ in $\Sigma_3$. A formula for $e_2$ is then
$$ e_2 = \frac{1}{3}(1 + [s])(1 - [t]) = \frac{1}{3}(1 + [s] - [t] - [st]).$$
Then one checks that $e_2V_*^{\otimes 3} = 0$ if $V_*$ is one dimensional, while $e_2V_*^{\otimes 3}$ is two dimensional if $V_*$ is two dimensional, and still larger if $V_*$ is larger.
\end{ex}

Our special idempotents are all examples of primitive idempotents in $\Q[\Sigma_k]$ that happen to lie in the subring $\Z_{(p)}[\Sigma_k]$. Standard representation theory tells us that primitive idempotents in $\Q[\Sigma_k]$ correspond to the irreducible $\Q[\Sigma_k]$--modules\footnote{More precisely, we mean equivalence classes of primitive idempotents under conjugation by units in the group ring.}.

These have been much studied for a long time, and we say a little bit about how the idempotents are constructed.  We will be a bit informal here, since \cite[Appendix C.1]{ravenel orange book} has precise details.  A useful textbook on this subject is \cite{james kerber}.

Firstly, the idempotents are parametrized by partitions of $k$, where a {\em partition} $\lambda$ of $k$ consists of a finite nonincreasing sequence of natural numbers $\lambda = (\lambda_1, \dots, \lambda_r)$ such that $\lambda_1 + \dots + \lambda_r = k$.

These are often pictured with Young diagrams, as in the next example.
\begin{ex}  The partition $\lambda = (4,2)$ of 6 has {\em Young diagram}
$$
\begin{array}{|c|c|c|c|}
\hline
&&& \\ \cline{1-4}
&& \multicolumn{1}{c}{} & \multicolumn{1}{c}{} \\ \cline{1-2}
\end{array}.
$$
\end{ex}

Now one fills the boxes in the Young diagram with the numbers 1, \dots, $k$: this is a {\em tableau} $t$ for $\lambda$ \cite[p.28]{james kerber}.

\begin{ex}  The `standard' tableau $t$ for $\lambda = (4,2)$ is
$$
\begin{array}{|c|c|c|c|}
\hline
1&2&3&4 \\ \cline{1-4}
5&6& \multicolumn{1}{c}{} & \multicolumn{1}{c}{} \\ \cline{1-2}
\end{array}
.$$
\end{ex}

Associated to a tableau $t$, one has row and column stabilizer subgroups of $\Sigma_k$ denoted $R_t$ and $C_t$.

\begin{ex} With $t$ as in the last example, $R_t = \Sigma_{\{1,2,3,4\}} \times \Sigma_{\{5,6\}}$ and $C_t = \Sigma_{\{1,5\}} \times \Sigma_{\{2,6\}}$ are the subgroups of $\Sigma_6 = \Sigma_{\{1,2,3,4,5,6\}}$.
\end{ex}

Also associated to the partition $\lambda$ is an integer $h_{\lambda}$, which is the product of the {\em hook lengths} associated to the boxes making up the Young diagram, where the hook length of a box is 1 + the number of boxes to the right and below the box.

\begin{ex}  If $\lambda = (4,2)$, we have hook lengths
$$
\begin{array}{|c|c|c|c|}
\hline
5&4&2&1 \\ \cline{1-4}
2&1& \multicolumn{1}{c}{} & \multicolumn{1}{c}{} \\ \cline{1-2}
\end{array}
$$
so $h_{\lambda} = 5\cdot 4 \cdot 2 \cdot 1 \cdot 2 \cdot 1 = 80$.
\end{ex}

Finally, given a tableau $t$ for $\lambda$, one defines $e_{t} \in \Q[\Sigma_k]$ by the formula
\begin{equation}  e_t = \frac{1}{h_{\lambda}}\left(\sum_{\alpha \in R_t} [\alpha]\right)\left(\sum_{  \beta \in C_t} \text{sgn}(\beta)[\beta]\right) = \frac{1}{h_{\lambda}}\sum_{\substack{\alpha \in R_t \\ \beta \in C_t}} \text{sgn}(\beta)[\alpha \beta].
\end{equation}
This turns out to be a primitive idempotent -- combine \cite[Theorem 3.1.10]{james kerber} with \cite[Theorem 2.3.21]{james kerber} --  and two different choices of tableau for $\lambda$ give conjugate idempotents.

Note that if a prime $p$ does not divide any of the hook lengths of $\lambda$, then $e_t \in \Z_{(p)}[\Sigma_k]$: the partition $\lambda$ is then called a {\em $p$--core} \cite[p.76]{james kerber}.

\begin{ex}  The partition $\lambda = (4,2)$ of 6 is a 3-core.
\end{ex}

There is one more construction we want to use that is not mentioned in \cite{ravenel orange book}: if $\lambda$ is a partition of $k$, there is an associated partition $\lambda^{\prime}$ obtained by interchanging the rows and columns in the Young diagram \cite[p.22]{james kerber}, and a tableau $t$ for $\lambda$ determines a tableau $t^{\prime}$ of $\lambda^{\prime}$.  Note that if $\lambda$ is a $p$--core, so is $\lambda^{\prime}$.

\begin{ex} If $\lambda = (4,2)$, with tableau
$ t =
\begin{array}{|c|c|c|c|}
\hline
1&2&3&4 \\ \cline{1-4}
5&6& \multicolumn{1}{c}{} & \multicolumn{1}{c}{} \\ \cline{1-2}
\end{array}
$,
then $\lambda^{\prime} = (2,2,1,1)$ and
$t^{\prime} =
\begin{array}{|c|c|}
\hline
1&5 \\ \cline{1-2}
2&6 \\ \cline{1-2}
3&\multicolumn{1}{c}{} \\ \cline{1-1}
4&\multicolumn{1}{c}{} \\ \cline{1-1}
\end{array}
$.
\end{ex}

From the construction, it is not hard to see that, if $V_*$ is a graded vector space over a field of characteristic 0, or $p$ if $\lambda$ is a $p$--core, there are isomorphisms of graded vector spaces
\begin{equation} \label{suspensions isomorphism} \Sigma^ke_{t^{\prime}}(V_*^{\otimes k}) \simeq e_t(\Sigma V_*^{\otimes k}).
\end{equation}
Here $\Sigma V_*$ denotes the graded vector space $V_*$ with a shift up by one in grading, so the even dimensional part of $V_*$ becomes the odd dimensional part of $\Sigma V_*$ and vice versa.

The idempotents in Propositions \ref{p=2 prop} and \ref{p odd prop} are all of the sort just described.

Given a prime $p$, let $\lambda_{p,d}$ be the partition of $(p-1)\binom{d+1}{2}$ given by
$$\lambda_{p,d} = ((p-1)d,(p-1)(d-1), \dots,(p-1)).$$
These are $p$--cores\footnote{Indeed, when $p=2$, these partitions are the only 2-cores.}: see \cite[Lemma C.1.4]{ravenel orange book}.  For all primes $p$, one now lets $e_d \in \Z_{(p)}[\Sigma_{(p-1)\binom{d+1}{2}}]$ be the idempotent $e_{t_{p,d}}$ where $t_{p,d}$ is the standard tableau for the partition $\lambda_{p,d}$.

At odd primes, we let $e_d^{\prime} = e_{t_d^{\prime}}$.

\cite[Thm.C.2.1]{ravenel orange book} says that, for all graded fields $\F_*$ of characteristic p, and all finite dimensional graded $\F_*$--vector spaces $V_*$,
\begin{equation} \label{prop equation} e_d V_*^{\otimes (p-1)k(d)} \neq 0 \Leftrightarrow (p-1)\dim_{\F_*} V_*^e + \dim_{\F_*} V_*^o \geq (p-1)d.
\end{equation}
This is given a careful proof in this reference.  Note how if $p=2$, this simplifies to the statement in \propref{p=2 prop}.  This $p=2$ statement can also be deduced directly from \cite[Corollary 8.1.17]{james kerber}.

When $p$ is odd, (\ref{suspensions isomorphism}) and (\ref{prop equation}) then combine to tell us that for all graded fields $\F_*$ of characteristic p, and all finite dimensional graded $\F_*$--vector spaces $V_*$,
\begin{equation} \label{prime equation} e_d^{\prime} V_*^{\otimes (p-1)k(d)} \neq 0 \Leftrightarrow (p-1)\dim_{\F_*} V_*^o + \dim_{\F_*} V_*^e \geq (p-1)d.
\end{equation}

It remains to explain why (\ref{prop equation}) and (\ref{prime equation}) also hold when the field $\F_*$ has characteristic 0, rather than $p$.

This follows immediately from the next lemma, noting that if $\F_*$ is a graded field of characteristic either 0 or $p$, then any graded $\F_*$--vector space $V_*$ can be written in the form $V_* = \F_* \otimes_{\Z_{(p)}} A_*$, where $A_*$ is a graded $\Z_{(p)}$--module, free in each degree.

\begin{lem} Let $A_*$ be a graded $\Z_{(p)}$--module, free in each degree, and let $K_*$ and $k_*$  respectively be graded fields of characteristic 0 and $p$.  If $e \in \Z_{(p)}[\Sigma_k]$ is any idempotent, the dimensions (over $K_*$) of the even and odd parts of $e(K_* \otimes_{\Z_{(p)}} A_*)^{\otimes k}$ will equal the dimensions (over $k_*$) of the even and odd parts of $e(k_* \otimes_{\Z_{(p)}} A_*)^{\otimes k}$.
\end{lem}
\begin{proof}
There are isomorphisms
$$ e(K_* \otimes_{\Z_{(p)}} A_*)^{\otimes k} \simeq K_* \otimes_{\Z_{(p)}} (eA_*^{\otimes k}),$$
and
$$e(k_* \otimes_{\Z_{(p)}} A_*)^{\otimes k} \simeq k_* \otimes_{\Z_{(p)}} (eA_*^{\otimes k}).$$
As $eA_*^{\otimes k}$ is a direct summand of $A_*^{\otimes k}$, and thus a free $\Z_{(p)}$--module, the dimensions for the even and odd parts of both  $e(K_* \otimes_{\Z_{(p)}} A_*)^{\otimes k}$ and $e(k_* \otimes_{\Z_{(p)}} A_*)^{\otimes k}$ agree with the ranks, over $\Z_{(p)}$, of the even and odd parts of $eA_*^{\otimes k}$.
\end{proof}

\begin{rem}  This elementary lemma has a curious consequence: if $V_*$ is a finite dimensional vector space over a graded field $\F_*$ of characteristic 0, and $e \in \Z_{(2)}[\Sigma_k]$ is any idempotent, then the total dimension of $e V_*^{\otimes k}$ is independent of the grading of $V_*$.  This is clearly true if $\F_*$ were of characteristic 2, and the lemma allows us to transfer this observation to characteristic 0.
\end{rem}

\section{Lower bounds for $r_n(G,H)$ using representation theory} \label{examples section}

In this section, we give the background needed to use our lens space and projective space constructions, and give the details of \exref{el ab example} (which implies that $r \leq r_n(C_p^r, \{e\})$), and then our more delicate \thmref{general lower bound theorem}.

\subsection{The fixed points of $L_p(\omega)$ and $\RP(\omega)$}

As in the introduction, if $\omega$ is a unitary representation of a finite group $G$, we let $L_p(\omega) = S(\omega)/C_p$, where $S(\omega)$ is the unit sphere in $\omega$, and $C_p \subset U(1) \subset \C^{\times}$ is the group of $p$th roots of unity.  Thus, if $\omega$ has complex dimension $d$, then $L_p(\omega) = L_p(\C^d)$ is a $(2d-1)$--dimensional lens space with an induced action of $G$.

\begin{lem} Given $x \in S(\omega)$, $[x] \in L_p(\omega)$ is fixed by $G$ if and only if $x$ spans a 1--dimensional sub-representation of $\omega$ which factors through $G/\Phi(G)$.
\end{lem}
\begin{proof}  Given $x \in S(\omega)$,  $[w] \in L_p(\omega)$ is fixed by $G$ if and only if we can define a character $\lambda: G \ra C_p \subset \C^{\times}$ by $gx=\lambda(g)x$, and such a character will factor through $G/\Phi(G)$.
\end{proof}

To use this lemma to describe $L_p(\omega)^G$ as a subspace of $L_p(\omega)$, we need to recall that $\omega$ admits a canonical decomposition into its isotypical components: $\displaystyle \omega \simeq \bigoplus_{i} \omega_i$,
with the sum running over an indexing set for the simple $\C[G]$--modules $\lambda_i$.  The summand $\omega_i$ is the span of all submodules isomorphic to $\lambda_i$.

Note that each inclusion $\omega_i \subset \omega$ induces a subspace inclusion $L_p(\omega_i) \subset L_p(\omega)$, and that these various subspaces are disjoint. The lemma thus has the following formula as a corollary.

\begin{prop} \label{lens space fixed point formula} $\displaystyle L_p(\omega)^G = \coprod_{i} L_p(\omega_i)$, with the disjoint union running over the 1--dimensional complex representations of $G$ which factor through $G/\Phi(G)$.
\end{prop}

Similarly, if $\omega$ is a real representation of a finite group $G$, we let $\RP(\omega)$ be the associated projective space: the $G$--space of real lines in $\omega$.  So if $\omega$ is $d$--dimensional, then $\RP(\omega) = \RP(\R^d)$ is a $(d-1)$--dimensional real projective space with an induced action of $G$.

\begin{lem} A point $L \in \RP(\omega)$ is fixed by $G$ if and only if, when viewed as a line in $\omega$, $L$ is a one dimensional sub-representation.
\end{lem}

Again one has an isotypical decomposition:
$\displaystyle \omega \simeq \bigoplus_{i} \omega_{i}$,
with the sum running over an indexing set for the simple $\R[G]$--modules $\lambda_i$.

The lemma thus has the following as a corollary.

\begin{prop} \label{projective space fixed point formula}$\displaystyle \RP(\omega)^G = \coprod_{i} \RP(\omega_i)$, with the disjoint union running over the 1--dimensional real representations of $G$.
\end{prop}

As in   \secref{general lower bound subsection}, given a finite 2--group $G$, we let $e_G \in \R[G]$ be the central idempotent
$$ e_G = \frac{1}{|\Phi(G)|}\sum_{g \in \Phi(G)} g.$$
One easily checks that $\Phi(G)$ acts trivially on the image of $e_G$, and that $e_G$ acts as the identity on any $\R[G]$--module on which $\Phi(G)$ acts trivially. It follows that if $\omega$ is a real representation of $G$, then the image of $e_G$ is the maximal direct summand of $\omega$ on which $\Phi(G)$ acts trivially, so can be viewed as a real representation of $G/\Phi(G)$ pulled back to $G$.

But this maximal direct summand is precisely $\displaystyle \bigoplus_i \omega_i$,
with the sum over the 1--dimensional real representations of $G$, as in the last proposition, and so we have the next corollary.

\begin{cor} \label{fixed point cor} The summand inclusion $e_G\omega \hookrightarrow \omega$ induces a homeomorphism $$\RP(e_G\omega)^G = \RP(\omega)^G.$$
\end{cor}

\subsection{The Morava $K$--theory of $L_p(\mathbb C^d)$ and $\RP(\mathbb R^d)$}

We will use the following familiar calculations.

\begin{prop} \label{RP prop} When $p=2$, one has
\begin{equation*}
k_n(\RP(\mathbb R^d)) =
\begin{cases}
d & \text{if } d \leq 2^{n+1} \\ 2^{n+1} & \text{if } d \text{ is even and } d \geq 2^{n+1} \\ 2^{n+1}-1 & \text{if } d \text{ is odd and } d > 2^{n+1}.
\end{cases}
\end{equation*}
\end{prop}

\begin{prop} \label{Lens space prop}  For all primes $p$, one has
\begin{equation*}
k_n(L_p(\mathbb C^d)) =
\begin{cases}
2d & \text{if } d \leq p^n \\ 2p^n & \text{if } d \geq p^n.
\end{cases}
\end{equation*}
\end{prop}

Both propositions follow from computations using the Atiyah--Hirzebruch spectral sequence converging to $K(n)^*(X)$. As already described in \secref{morava k theory background section}, this has $E_2^{*,*} = H^*(X;\Z/p)[v_n^{\pm 1}]$, and first possible nonzero differential given by $ d_{2p^n-1}(x) = v_nQ_n(x)$,  where $Q_n$ is the $n$th Milnor primitive in the mod $p$ Steenrod algebra.

For the first proposition, one knows that $H^*(\RP(\R^d);\Z/2) = \Z/2[x]/(x^d)$.  From the definition of $Q_n$, one sees that $Q_n(x) = x^{2^{n+1}}$ if $2^{n+1} < d$.  As $Q_n$ is a derivation, it follows that $Q_n(x^r) = x^{2^{n+1}+r}$ if $r$ is odd and $2^{n+1} + r< d$,  and is 0 otherwise.  It follows that nonzero elements in the $E_{2^{n+1}}$--term of the spectral sequence consists of the even dimensional classes between degrees 0 and $2^{n+1}$, and the odd dimensional classes between degrees $d+1-2^{n+1}$ and $d$. Even dimensional classes are in the image of $K(n)^*(\C \mathbb P^{\infty})$ under the composite $\RP(\R^d) \hra \RP^{\infty} \ra \C \mathbb P^{\infty}$, and so are permanent cycles.  As the odd dimensional elements in the $E_{2^{n+1}}$--term are clearly permanent cycles for dimension reasons, it follows that there can be no higher differentials, and the proposition follows by counting classes.

When $p=2$, the first proposition includes the second, as $L_2(\mathbb C^d) = \RP(\R^{2d})$.

The proof of the second proposition when $p$ is odd is similar, starting from the calculations $H^*(L_p(\C^d);\Z/p) = \Lambda^*(x) \otimes \Z/p[y]/(y^d)$, with  $Q_n(x) = y^{p^n}$ if $p^n < d$.

\subsection{Using $L_p(\omega)$: the details of \exref{el ab example}}

Recall the situation of \exref{el ab example}.  Let $E_r = (C_p)^r$ and let $\rho^{\C}_r$ denote its regular representation: the sum of the $p^r$ distinct 1--dimensional complex representations of $E_r$.

We let $\omega = p^n\rho_r^{\C}$, and we want to show that
$$  k_{n+r-1}(L_p(\omega)) = 2p^{n+r-1}$$
and
$$  k_n(L_p(\omega)^{E_r}) = 2p^{n+r}. $$

As $p^n\rho_r^{\C}$ is $p^{n+r}$ dimensional, $L_p(\omega) = L_p(\C^{p^{n+r}})$, and the first follows immediately from \propref{Lens space prop}.

As $\omega$ is the direct sum of $p^n$ copies of each of the $p^r$ 1--dimensional representations of $E_r$, \propref{lens space fixed point formula} tells us that $L_p(\omega)^{E_r}$ is the disjoint union of $p^r$ copies of $L_p(\C^{p^n})$, and the second calculation also follows from \propref{Lens space prop}.

\subsection{Proof of \thmref{general lower bound theorem}}

We recall the hypotheses of \thmref{general lower bound theorem}.

We are assuming that $H$ be a proper nontrivial subgroup of a finite 2--group $G$ such that $\Phi(H) = \Phi(G)\cap H$, and that  $G$ has an irreducible real representation $\Delta$  such that $e_H\Res^G_H(\Delta)$ is the regular real representation of $H/\Phi(H)$ pulled back to $H$.

We wish to show that then, for all $n$, $ r_n(G,H) \geq r_-(G,H) + 1$.

Let $a=$ the rank of $G/\Phi(G)$ and let $b=$ the rank of $H/\Phi(H)$. The hypothesis $\Phi(H) = \Phi(G)\cap H$ means that $H/\Phi(H) \ra G/\Phi(G)$ is monic, and thus $r_-(G,H) = a-b$.  So we wish to show that $r_n(G,H) \geq a-b+1$.

Let $\widetilde \rho_G$ be the regular representation of $G/\Phi(G)$ pulled back to $G$, and, similarly, let $\widetilde \rho_H$ be the regular representation of $H/\Phi(H)$ pulled back to $H$.

Fixing $n$, let $$\omega = 2^{n+1} \widetilde \rho_G \oplus \Delta.$$

By \corref{lower bound method cor}, $r_n(G,H) \geq a-b+1$ will follow if we can show that
$$ k_{n+a-b}(\RP(\omega)^H) < k_n(\RP(\omega)^G).$$

The next two lemmas say what we need.

\begin{lem} \label{Morava K-theory of G fixed point lemma}  $k_n(\RP(\omega)^G) = 2^{n+a+1}$
\end{lem}

\begin{lem} \label{Morava K-theory of H fixed point lemma} $k_{n+a-b}(\RP(\omega)^H) = 2^{n+a+1} -2^b$.
\end{lem}

\begin{proof}[Proof of \lemref{Morava K-theory of G fixed point lemma}]  Each of the $2^a$ 1--dimensional real representations of $G$ occurs exactly $2^{n+1}$ times in $\omega$, and thus the corresponding isotypical components of $\omega$ all have dimension $2^{n+1}$.  Thus
$$ k_n(\RP(\omega)^G) = 2^a k_n(\RP(\R^{2^{n+1}})) = 2^a \cdot 2^{n+1} = 2^{n+a+1}.$$
\end{proof}

\begin{proof}[Proof of \lemref{Morava K-theory of H fixed point lemma}]
By \corref{fixed point cor}, $\RP(\omega)^H = \RP(e_H\omega)^H$.  We analyze $e_H\Res^G_H(\omega)$.

Since $H/\Phi(H)$ has index $2^{a-b}$ in $G/\Phi(G)$, we have that $\Res^G_H(\widetilde \rho_G) = 2^{a-b}\widetilde \rho_H$, and $e_H$ acts as the identity on this.   Meanwhile, we have assumed that $e_H\Res^G_H(\Delta) = \widetilde \rho_H$.  Thus
\begin{equation*}
\begin{split}
e_H\Res^G_H(\omega) &
= e_H\Res_H^G(2^{n+1}  \widetilde \rho_G \oplus \Delta) \\
  & = (2^{n+1}2^{a-b}+1)\widetilde \rho_H \\
  & = (2^{n+a-b+1}+1)\widetilde \rho_H.
\end{split}
\end{equation*}
This implies that each of the $2^b$ 1--dimensional representations of $H$ occurs $2^{n+a-b+1}+1$ times in $\omega$, and thus the corresponding isotypical components of $\omega$ all have dimension $2^{n+a-b+1}+1$.  Thus
$$ k_{n+a-b}(\RP(\omega)^H) = 2^b k_n(\RP(\R^{2^{n+a-b+1}+1})) = 2^b \cdot (2^{n+a-b+1}-1) = 2^{n+a+1}-2^b.$$

\end{proof}

\begin{ex}  It is worth seeing explicitly how and why this all works in the simplest example, when $n=0$, $G = D_8$ and $H = C$, a noncentral subgroup of order 2.

$D_8$ has four 1-dimensional representations $\lambda_1, \dots, \lambda_4$, and one 2--dimensional irreducible $\Delta$, which, when restricted to $C$ is $1 \oplus \sigma$, the sum of the two 1-dimensional representations of $C$.

We let
$$\omega = 2(\lambda_1 \oplus \lambda_2 \oplus \lambda_3 \oplus \lambda_4) \oplus \Delta,$$
a 10--dimensional representation of $D_8$.  Thus $\RP(\omega)$ is the space  $\RP(\R^{10}) = \RP^9$ with an action of $D_8$, and we have
$$ (\RP^9)^{D_8} = \coprod_4 \RP(\R^2) =\coprod_4 \RP^1, $$
So that $k_0(\RP^9)^{D_8} = 4\cdot 2 = 8$.

Meanwhile, when viewed as a representation of $H$,
$$ \omega = 5(1 \oplus \sigma).$$
Thus
$$(\RP^9)^{C} = \coprod_2 \RP(\R^5) = \coprod_2 \RP^4.$$
Since $Q_1$ acts nontrivially on $H^*(\RP^4;\Z/2)$ (not true with $\RP^4$ replaced by $\RP^3$!), we have that $k_1(\RP^4) = 3$, so that $k_1(\RP^9)^{C} = 2\cdot 3 = 6$.

\end{ex}

\section{Blue shift numbers for extraspecial 2-groups} \label{ex special 2-groups section}

\subsection{Extraspecial 2--groups and their real representations}

We collect some information we will need about extraspecial 2--groups and their real representations.  A general reference for the group theory is \cite[Chapter 8]{aschbacher}, and \cite{quillen} has what we need about the representation theory.

By definition, an extraspecial 2--group is a finite 2--group $\widetilde E$ such that $\widetilde E^{\prime} = \Phi(\widetilde E) = Z(\widetilde E)$ is cyclic of order 2.  Thus it is a nonabelian group that fits into a central extension
$$ C_2 \xra{i} \widetilde E \xra{\pi} E,$$
with $E$ elementary abelian.

One defines $q: E \ra C_2$ by the formula $q(a) = c$ if $\pi(\tilde a) = a$ and $\tilde a^2 = i(c)$, and then $\langle \text{\hspace{.1in}, \hspace{.1in}} \rangle: E \times E \ra C_2$ by $\langle a, b \rangle = q(a+b) - q(a) - q(b)$.  Then $\langle \text{\hspace{.1in}, \hspace{.1in}} \rangle$ is nondegenerate, symmetric, and bilinear, and $q$ is a quadratic form.  These then determine the group structure on $\widetilde E$ by the formulae $\tilde a^2 = i(q(\pi(\tilde a)))$ and $[\tilde a, \tilde b] = i(\langle \pi(\tilde a), \pi(\tilde b)\rangle)$.

Quadratic forms like this are classified by their Arf invariant: one learns that $E$ must be of even dimension, and that, up to isomorphism,  there are two distinct possible quadratic functions $q$ on $E_{2r}$. The one that will concern us has Arf invariant 0: $q = x_1y_1+ \dots + x_ry_r$, where $(x_1, \dots, x_r, y_1, \dots, y_r)$ is dual to a basis $(a_1, \dots, a_r, b_1, \dots, b_r)$ for $E$.

It follows that $\widetilde E_{2r}$ has generators $\tilde a_1, \dots, \tilde a_r, \tilde b_1, \dots, \tilde b_r, c$ with $c^2 = e$, $\tilde a_i^2 = \tilde b_i^2 = c$ for all $i$, and with all generators commuting except that $[\tilde a_i,\tilde b_i] = c$ for all $i$.

A subspace $W < E_{2r}$ is {\em isotropic} if $\langle W,W \rangle = 0$.  It is not hard to see that, under $\pi: \widetilde E_{2r} \ra E_{2r}$, maximal elementary abelian subgroups of $\widetilde E_{2r}$ not containing the center $\langle c \rangle$ will correspond to maximal isotropic subspaces of $E_{2r}$, and all such will be equivalent to $W_r$, the subgroup generated by $\tilde a_1, \dots, \tilde a_r$.

Our group $\widetilde E_{2r}$ is sometimes denoted $2_+^{1+2r}$ in the literature, and can also be described as $D_8^{\circ r}$, the central product of $r$ copies of the dihedral group $D_8$ of order 8.  The other extraspecial 2--group of order $2^{1+2r}$, sometimes denoted $2_-^{1+2r}$, is $Q_8 \circ D_8^{\circ r-1}$, where $Q_8$ is the quaternionic group of order 8.

The $2^r$ 1--dimensional real representations of $E_{2r}$ pullback to give 1--dimensional real representations of $\widetilde E_{2r}$.  The group $\widetilde E_{2r}$ has one more irreducible real representation $\Delta_r$, a faithful representation of dimension $2^r$ on which $c$ acts as $-1$.  Of key importance to us is that $\Delta_r$ restricted to $W_r$ is the regular representation of $W_r$ \cite[(5.1)]{quillen}.

\subsection{The computation of $r_n(\widetilde E_{2r}, W_r)$ and $r_n(\widetilde E_{2r} \times E_s,W_r \times \{e\})$}

We compute $r_n(\widetilde E_{2r}, W_r)$. We have that
$$\Phi(W_r) = \{e\} = \Phi(\widetilde E_{2r}) \cap W_r,$$
and $\Delta_r$ restricted to $W_r$ is the regular representation, so the hypotheses of \thmref{general lower bound theorem} hold.  As $r_-(\widetilde E_{2r}, W_r) = r$, we deduce the conclusion of \thmref{extraspecial group thm}: $r_n(\widetilde E_{2r}, W_r) = r+1= r_+(\widetilde E_{2r}, W_r)$.

To fill in the details of \exref{extraspecial 2 group ex}, we let $\omega = 2^{n+1}\widetilde \rho_{2r} \oplus \Delta_r$,
where $\widetilde \rho_{2r}$ is the real regular representation of $E_{2r}$, pulled back to $\widetilde E_{2r}$.

\lemref{Morava K-theory of G fixed point lemma} tells us that
$$ k_n(\RP(\omega)^{\widetilde E_{2r}}) = 2^{n+1+2r}, $$
while \lemref{Morava K-theory of H fixed point lemma} tells us that
$$  k_{n+r}(\RP(\omega)^{W_r}) = 2^{n+1+2r}-2^r.$$

Similarly, \thmref{extraspecial group product with el ab thm} is the special case of \thmref{general lower bound theorem}, applied to the pair $(\widetilde E_{2r} \times E_s, W_r \times \{e\})$, with the special representation of $\widetilde E_{2r} \times E_s$ chosen to be $\Delta_r$, pulled back to the product.  Now $r_-(\widetilde E_{2r} \times E_s, W_r \times \{e\} = r+s$, so we learn that
$$r_n(\widetilde E_{2r} \times E_s,W_r \times \{e\}) = r+s+1  = r_+(\widetilde E_{2r}\times E_s,W_r \times \{e\}).$$

\subsection{Blue shift numbers for a family of groups}

Let $\mathcal G$ be the collection of 2--groups $G$ fitting into a central extension
$$ C_2 \ra G \xra{p} E$$
with $E$ elementary abelian.

Our goal in this subsection is to prove \thmref{good family thm}, which said the following:

{\em Let $G \in \mathcal G$.  For all $K<H<G$, $r_n(H,K) = r_+(H,K)$ for all $n$.} \\

We show that the family of calculations
$$r_n(\widetilde E_{2r} \times E_s,W_r \times \{e\}) = r+s+1  = r_+(\widetilde E_{2r}\times E_s,W_r \times \{e\})$$
suffices to prove this.

We start with some elementary observations.

Note that if $G$ is in $\mathcal G$ and $H<G$ is nontrivial, then $H$ is again in $\mathcal G$.  Thus it suffices to show that if $G \in \mathcal G$ then $r_n(G,H) = r_+(G,H)$ for all $H<G$.  We prove this by induction on $|G|$.

For any group $G \in \mathcal G$, either $G$ is abelian or $G^{\prime} = \Phi(G) = C_2$. If $G$ is abelian we are done: $r_n(G,H) = r_+(G,H)$ for all $H<G$.  Thus we can assume this is not the case.

Next observe that if $G \in \mathcal G$ and $N\triangleleft G$ is any proper normal subgroup, then $G/N \in \mathcal G$.  Now let $N = H \cap Z(G)$ which will be a normal subgroup of $G$.  Then $r_n(G,H) = r_n(G/N,H/N)$ and $r_+(G,H) = r_+(G/N,H/N)$.  If $N \neq \{e\}$ then $r_n(G/N,H/N) = r_+(G/N,H/N)$ by our inductive assumption, and we are done.  Thus we can assume that $H \cap Z(G) = \{e\}$.

Since $\Phi(G) \leq Z(G)$, we see that $H \cap \Phi(G) = \{e\}$ also.  This implies that $p: H \ra E$ is monic, so $H$ is elementary abelian.

We isolate the next part of our argument as a lemma.

\begin{lem} \label{order 4 lemma}  In this situation, suppose that $C_G(H)$ contains an element of order 4.   Then $r_-(G,H) = r_+(G,H)$, and so $r_n(G,H) = r_+(G,H)$.
\end{lem}
\begin{proof}
Suppose that there exists $x \in C_G(H)$ of order 4.  As $H$ is elementary abelian, we know that $x \not \in H$. Since $x^2$ must generate $\Phi(G)$, we can further conclude that $x \not \in H\Phi(G)$, which means that $p(x) \not \in p(H)$.  If we let $K<G$ be the subgroup generated by $H$ and $x$, then we have
\begin{equation*}
\begin{split}
r_+(G,H)  &
\leq 1 + r_+(G,K) \\
  & = 1 + r_+(E, p(K)) \\
  & = 1 + r_-(E, p(K)) \\
  & = r_-(E, p(H)) \\
  & = r_-(G, H),
\end{split}
\end{equation*}
\end{proof}

Since $Z(G) < C_G(H)$, the lemma implies that we can assume that $Z(G)$ is elementary abelian, and thus admits a decomposition $Z(G) = C_2 \times E_s$ for some $s$, where the first factor is $G^{\prime}$.  If we let $\widetilde E = G/E_s$ then $\widetilde E$ will be an extraspecial 2--group, so $\widetilde E/\widetilde E^{\prime} = E_{2r}$, for some $r$, and the sequence
$$ C_2 \ra G \xra{p} E$$
identifies with a sequence of the form
$$ C_2 \ra \widetilde E \times E_s \xra{p \times 1} E_{2r} \times E_s.$$

Now recall that $H < G = \widetilde E \times E_s$ is elementary abelian and that $H \cap Z(G)$ is trivial.  Since $Z(G) = C_2 \times E_s$, we conclude that $H$ projects isomorphically to an elementary abelian subgroup in $\widetilde E$ that does not contain the central $C_2$. Another way of putting this, is that $H$ is the graph of a homomorphism $W \ra E_s$, where $W < \tilde E$ is an elementary abelian subgroup that does not contain the central $C_2$.  One can conclude that $C_G(H) = C_G(W) = C_{\widetilde E}(W) \times E_s$.

If $\widetilde E = Q_8 \circ D_8^{\circ r-1}$, then $Q_8 < C_{\widetilde E}(W)$, and so the centralizer contains an element of order 4, and the lemma applies.  Similarly, there is an element of order 4 in $C_{\widetilde E}(W)$  if $\widetilde E = \widetilde E_{2r}$ and $W$ has rank less than $r$.

So we can assume that our pair $(G,H) = (\widetilde E_{2r} \times E_s, H)$ where $H$ is the graph of a homomorphism $W_r \ra E_s$.  But it is easy to check that this pair is equivalent to $(\widetilde E_{2r} \times E_s, W_r \times \{e\})$, so $r_n(G,H) = r_+(G,H)$ by \thmref{extraspecial group product with el ab thm}.

\section{Essential pairs and final remarks} \label{essential pairs section}

\subsection{Essential pairs}

If one wishes to systematically try to prove that $r_n(G,H) = r_+(G,H)$ for all $H < G$, one can focus on potential minimal counterexamples.  By pulling back actions through quotient maps, these must be pairs as in the following definition.

\begin{defn} If $H$ is a subgroup of a finite $p$--group $G$, say $(G,H)$ is an {\em essential pair}, if $r_+(G,H) > r_+(G/N,HN/N)$ for all nontrivial normal subgroups $N \lhd G$.
\end{defn}

Clearly if $(G,H)$ is essential, it is necessary that $H$ contain no nontrivial normal subgroups of $G$, so, in particular, $H \cap Z(G) = \{e\}$.  Also, since $r_-(G,H) = r_+(G/\Phi(G), H\Phi(G)/\Phi(G))$, we see that if $(G,H)$ is essential, then either $r_-(G,H)<r_+(G,H)$ or $\Phi(G)$ is trivial (i.e., $G$ is elementary abelian).

To easily identify essential pairs, the following lemma is useful.

\begin{lem} $(G,H)$ is essential if $r_+(G,H) > r_+(G/C,HC/C)$ for all central subgroups $C<G$ of order $p$.
\end{lem}
\begin{proof}  We prove the lemma in its contrapositive form.

The group $G$ acts on any normal subgroup $N$ by conjugation, and the fixed point set identifies with $N \cap Z(G)$.  As $G$ is a $p$--group, the number of fixed points must be congruent to 0 mod $p$. As $e \in N$ is clearly fixed, we conclude that $N \cap Z(G)$ is nontrivial, and thus $N$ contains a central subgroup $C$ of order $p$.

Since $r_+(G,H) \geq r_+(G/C,HC/C) \geq r_+(G/N,HN/N)$ holds in general, if $r_+(G,H) = r_+(G/N,HN/N)$ then $r_+(G,H) = r_+(G/C,HC/C)$.
\end{proof}

\begin{exs}

The only essential pairs $(G,H)$ with $G$ abelian are the pairs $(E_r, \{e\})$, for $r \geq 1$.
\end{exs}

\begin{exs} At the prime 2, the pairs $(\widetilde E_{2r} \times E_s,W_r \times \{e\})$ are essential, and, up to equivalence, there are no other essential pairs $(G,H)$ with $G$ in the family of groups for which \thmref{good family thm} applies.
\end{exs}

\begin{exs} If $p$ is odd, let $\widetilde E_{2r}$ denote the extraspecial group of order $p^{1+2r}$ having exponent $p$.  Just as in the $p=2$ case, we let $W_r$ denote any elementary abelian $p$--subgroup of rank $r$ that doesn't contain the center.   Then $(\widetilde E_{2r} \times E_s, \{e\})$ is essential, with
$$r_{-}(\widetilde E_{2r} \times E_s,\{e\}) = 2r + s < 1+2r+s = r_{+}(\widetilde E_{2r} \times E_s,\{e\}),$$
as are the pairs $(\widetilde E_{2r} \times E_s, W_r \times \{e\})$, with
$$r_{-}(\widetilde E_{2r} \times E_s, W_r \times \{e\}) = r + s < 1+r+s = r_{+}(\widetilde E_{2r} \times E_s, W_r \times \{e\}).$$
\end{exs}

In the appendix, we include tables of all essential pairs $(G,H)$ with $G$ a nonabelian 2--group of order up to 32.  Here we highlight a few of these that we feel are the pairs that need to be understood if any more significant progress is to be made on the Chromatic Smith Theorem problem.

\begin{ex}  Let $G$ be the group with GAP label 16 \#3.  This is a semidirect product $C_2^2\rtimes C_4$, with $C_4$ acting on $C_2^2$ via the quotient $C_4 \twoheadrightarrow C_2$.  It also fits into a central extension
$$ C_2 \ra G \ra C_2 \times C_4,$$
so the group is `almost' in our family of friendly groups dealt with in \thmref{good family thm}, but not quite.

Then $(G,\{e\})$ is essential, $r_-(G,\{e\}) = 2$ and $r_+(G,\{e\}) = 3$.
\end{ex}

\begin{ex}  Let $M_2(4)$ be the group with GAP label 16 \#6.  This fits into a central extension
$$ C_4 \ra M_2(4) \ra C_2 \times C_2,$$
and also a central extension
$$ C_2 \ra M_2(4) \ra C_2 \times C_4,$$
so again the group is almost, but not quite,  in our family of friendly groups.

Let $C<M_2(4)$ be any of the noncentral subgroups of order 2, e.g. GAP subgroup \#3.
Then $(M_2(4),C)$ is essential, $r_-(M_2(4),C) = 1$ and $r_+(M_2(4),C) = 2$.

It is interesting to note that, though $(M_2(4),C)$ is essential, $(M_2(4)\times C_2,C \times \{e\})$ is not, as the calculation in \exref{M_4(2) product example} shows.
\end{ex}

\begin{ex}  The dihedral group $D_{16}$ has GAP label 16 \#7.  Then $Z(D_{16}) = C_2 < C_4 = D_{16}^{\prime} = \Phi(D_{16})$, so $D_{16}$ can be written as a noncentral extension
$$ C_4 \ra D_{16} \ra C_2 \times C_2.$$

Let $C<D_{16}$ be any of the noncentral subgroups of order 2, e.g. GAP subgroup \#3.
Then $(D_{16},C)$ is essential, $r_-(D_{16},C) = 1$ and $r_+(D_{16},C) = 3$.
\end{ex}

Our last example illustrates how complicated things become as one examines groups of order 32.

\begin{ex}  Let $G$ be the group with GAP label 32 \#6.  This is a semidirect product $C_2^3\rtimes C_4$, with $C_4$ acting faithfully on $C_2^3$.

Then $(G,H)$ is essential when $H$ is any of the inequivalent subgroups with GAP number \#3,9,14,19,24.
\end{ex}

\subsection{The limitations of our $\RP(\omega)$ and $L_p(\omega)$ constructions}

It is worth pondering why we are able to prove interesting lower bound theorems using the $\RP(\omega)$ and $L_p(\omega)$ constructions, and how these theorems are limited.

Informally, the fixed point formulae for $\RP(\omega)^G$ shows that there is a mod 2 cohomology class for each 1--dimensional real representation in $\omega$, and these are arranged in `piles' corresponding to the distinct representations. When one considers $\RP(\omega)^H$, these piles get `stacked up' when distinct representations becomes the same when restricted to $H$, and there are new classes coming from higher dimensional irreducible summands of $\omega$ that have 1--dimensional summands when restricted to $H$.

In the proof of the elementary abelian lower bound, \thmref{el ab lower bound theorem}, the action of the Milnor $Q_m$'s on the piles align just right to show that $r_-(G,H) \leq r_n(G,H)$.  To do better, we need $k_{m}(\RP(\omega)^H)$ to be made smaller, and this means that we need some new 1-dimensional $H$--representations to cancel some of those pulled back from $G$, via the operation $Q_m$.  Since, as a function of $d$, $k_m(\RP(\R^d))$ goes up and down only by 1 once $d$ is large, we see that $|H/\Phi(H)|$, the number of 1--dimensional representation of $H$, is the most that we can lower $k_{m}(\RP(\omega)^H)$, by adding higher dimensional irreducible  $G$--representations to $\omega$. When one ponders the numbers, it becomes clear that this is not enough of a change to prove more than $r_-(G,H)+1 \leq r_n(G,H)$.

In the odd prime situation, the situation is even worse: a 1-dimensional complex representation of $H$ contributes both an odd and an even dimensional class to the mod $p$ cohomology of $L_p(\omega)^H$, and if we add such a pair like this coming from a new 1--dimensional $H$--representation, $Q_m$ may pair an old odd class with the new even class, but the new odd class will still be left.  Otherwise said, $k_m(L_p(\C^d))$ is constant once $d$ is large, and we can never use this method to prove more than $r_-(G,H) \leq r_n(G,H)$.

\appendix

\section{Essential pairs $(G,H)$ with $|G| \leq 32$}

\begin{center}
\begin{tabular}{|l | l | c | c |  c |}
  \multicolumn{5}{c}{{\bf Essential Pairs for Nonabelian Groups of Order 8 and 16}}\\ \hline
  \(G\) \hfill [GAP label]& \(H\) [GAP subgroup]& \(r_-(G,H)\) & \(r_+(G,H)\)
  & \(r_n(G,H)\)\\ \hline

\(D_8\) \hfill [8, 3]&$C_2$ [3]&1&2& 2 \\
\( C_2^2 \rtimes C_4 \) \hfill [16, 3]&\{e\} [1] &2&3& \\
\(M_4(2)\) \hfill [16, 6]& $C_2$ [3] &1&2& \\
\(D_{16}\) \hfill [16, 7]& $C_2$ [3] &1&3& $\geq 2$ \\
\(SD_{16}\) \hfill  [16, 8]& $C_2$ [3] &1&3& $\geq 2$ \\
\(C_2 \times D_8\) \hfill  [16, 11]& $C_2$ [3] &2&3& 3 \\  \hline
\end{tabular}
\vspace{.2in}

\begin{tabular}{|l | l | l | l |  l |}
  \multicolumn{5}{c}{\bf Essential Pairs for Nonabelian Groups of Order 32}\\ \hline
  \(G\) \hfill [GAP label]& \(H\) [GAP subgroup]& \(r_-(G,H)\) & \(r_+(G,H)\)
  & \(r_n(G,H)\)\\ \hline
\(C_2^3 \rtimes C_4\)  \hfill [32, 6]&$C_2$ [3], $C_2^2$ [14], $C_2^2$ [19]&2, 2, 1&3, 3, 2&  \\
     &$C_2$ [9], $C_4$ [24]&1, 1&3, 3&$\geq 2$, $\geq 2$ \\
\(C_4.D_8\)  \hfill [32, 7]&$C_2$ [3]&1&3& $\geq 2$ \\
     &$C_2$ [7], $C_2^2$ [20]& 2, 1& 3, 2&  \\
\(C_4 \wr C_2\)  \hfill [32, 11]&$C_2$ [3]&1&3& $\geq 2$ \\
     &$C_2$ [7]& 2& 3&  \\
\(M_5(2)\)  \hfill [32, 17]&$C_2$ [3]&1&2& \\
\(D_{32}\)  \hfill [32, 18]&$C_2$ [3]&1&4&  $\geq 2$ \\
\(SD_{32}\)  \hfill [32, 19]&$C_2$ [3]&1&4& $\geq 2$ \\
\(C2 \times (C_2^2 \rtimes C_4)\)  \hfill [32, 22]&\{e\} [1]&3&4& \\
\(C_2^2 \wr C_2\)  \hfill [32, 27]&\{e\} [1], $C_2^2$ [46]&3, 1&4, 3& \\
     &$C_2$ [7]&2& 4& $\geq 3$\\
\(C_4^2\rtimes C_2\)  \hfill [32, 33]&$C_2$ [5]&2&3& \\
\( C_4 \rtimes D_8 \)  \hfill [32, 34]& $C_2$ [5]&2&4& $\geq 3$\\
\(C_2 \times D_{16}\)  \hfill [32, 39]&$C_2$ [3]&2&4&  $\geq 3$ \\
\(C_2 \times QD_{16}\)  \hfill [32, 40]&$C_2$ [3]&2&4& $\geq 3$\\
\(C_8 \rtimes C_2^2 \)  \hfill [32, 43]&$C_2$ [3], $C_2^2$ [28]&2, 1&4, 3& $\geq 3$, $\geq 2$\\
\(C_2^2 \times D_8\)  \hfill [32, 46]&$C_2$ [3]&3&4& 4 \\
\(D_8 \circ D_8\)  \hfill [32, 49]&$C_2^2$ [32]&2&3& 3  \\ \hline
\end{tabular}

\end{center}

\end{document}